\pdfoutput=1
\documentclass[3p,times,numbers,square]{elsarticle}
\usepackage{amsmath, amssymb, amsfonts} 
\usepackage{booktabs}
\usepackage{graphicx}
\usepackage{amsthm}
\usepackage{microtype}
\usepackage{hyperref}
\usepackage{multicol}
\usepackage{multirow}
 \usepackage{tikz}
\usepackage{mathpazo} 
\theoremstyle{definition}
\newtheorem{theorem}{Theorem}
\theoremstyle{definition}
\newtheorem{remark}{Remark}
\begin{document}
\begin{frontmatter}
\title{A Linear \texorpdfstring{$\alpha$}{alpha}-cut intervals based Parsimonious Fuzzy Best-Worst Method with an Application to Warehouse Location Selection}

\author{Vikas V. Sharma}
\ead{vikas.sharma.iitram@gmail.com}
\author{Mohit Kumar\corref{cor1}}
\ead{mohitkumar@iitram.ac.in}
\address{Department of Basic Sciences,\\ Institute of Infrastructure, Technology, Research And Management, Ahmedabad,\\ Gujarat-380026, India\\
Email: vikas.sharma.iitram@gmail.com, mohitkumar@iitram.ac.in}
\cortext[cor1]{Corresponding author}

\begin{abstract}
To address the limitations of existing fuzzy extensions of the Best-Worst Method (BWM), Ratandhara and Kumar proposed a model based on the combination of fuzzy sets and $\alpha$-cut intervals, known as the $\alpha$-cut intervals based Fuzzy Best-Worst Method ($\alpha$-FBWM). Despite its conceptual strengths, this approach faces significant execution challenges. It requires solving $2n+1$ complex non-linear optimization models, followed by subsequent interval boundary averaging to determine a unique weight set for $n$ decision criteria. This makes the method computationally intensive and time-consuming. Furthermore, because final unique weights are generated as crisp numbers, this makes it impossible to capture the inherent uncertainty within the obtained weights. To resolve these limitations, this research proposes a novel, computationally efficient framework named the Linear $\alpha$-cut intervals based Fuzzy Best-Worst Method (Linear $\alpha$-FBWM). The proposed approach reformulates a non-linear optimization model into a single linear programming model. This allows the framework to directly find the unique criteria weights as Triangular Fuzzy Numbers (TFNs), which effectively capture uncertainties in the resulting weights. Additionally, to evaluate the ordinal consistency of the generated weights, a metric named the Ordinal Preference Violation (OPV) is introduced based on the cognitive psychology-driven prominence effect. This metric assesses the alignment between the final weight rankings and the decision-maker's initial fuzzy pairwise preferences. Numerical examples show that the Linear $\alpha$-FBWM minimizes these logical violations and performs comparably to or better than the previous $\alpha$-FBWM model.
\par
To efficiently handle large-scale applications with numerous alternatives, we extend this framework into a new model, the Linear $\alpha$-cut intervals based Parsimonious Fuzzy Best-Worst Method (Linear $\alpha$-PFBWM), which directly integrates our proposed Linear $\alpha$-FBWM formulation. This parsimonious extension represents a major improvement over existing parsimonious extensions of FBWM, because the initial alternative ratings can be directly given as fuzzy numbers, and the final priorities of non-reference alternatives are calculated through fuzzy interpolation without any early defuzzification of the priorities of reference alternatives. The step-by-step computational process of the Linear $\alpha$-PFBWM is validated through a numerical example from existing literature to verify its scaling capabilities.
 Finally, the practical utility of the proposed Linear $\alpha$-PFBWM is confirmed using a real-world strategic industrial case study, selecting the ideal warehouse locations from 20 alternatives across Gujarat for a leading multinational paint manufacturing company. Using this integrated approach, the expert only required 20 initial fuzzy ratings and 35 fuzzy pairwise comparisons, instead of the usual 185. In the case of intense fuzzy pairwise comparisons, this constitutes a substantial decrease of $81.08\%$, proving the framework to be an efficient, time-saving, and reliable tool for complex industrial decision-making.
\end{abstract}

\end{frontmatter}

\section{Introduction}
In everyday life, individuals and organizations consistently encounter difficult decisions. Examples include parents selecting an educational institution for their child, a family deciding to purchase a new car, or a corporation evaluating multiple business projects. In all such scenarios, the objective is to select the optimal alternative or to rank multiple options based on conflicting criteria. To systematically address these complex problems, famous Multi-Criteria Decision-Making (MCDM) methods such as AHP \cite{saaty1990make}, ANP \cite{ANP}, TOPSIS \cite{TOPSIS}, VIKOR \cite{VIKOR}, ELECTRE \cite{ELECTRE}, and PROMETHEE \cite{PROMETHEE} have traditionally been utilized by researchers and practitioners.

However, a key practical challenge arises with pairwise comparison based methods like AHP. In this method, the Decision-Maker (DM) constructs a complete pairwise comparison matrix to compute the criteria weights. More precisely, this involves \texorpdfstring{$n(n - 1)/2$}{n(n-1)/2} pairwise comparisons for $n$ criteria. While this method is effective for simple cases, but as the number of criteria or alternatives increases, the evaluation process quickly becomes tedious, time-consuming, and mentally exhausting. Naturally, handling such a massive cognitive load overwhelms DMs, causing them to lose focus and leading to inconsistent and unreliable judgments.

The Best-Worst Method (BWM) was developed as a highly efficient solution to address this significant challenge \cite{rezaei2015best}. BWM reduces the complexity of decision-making by using a reference based approach. Instead of comparing every single criterion against all others, the DM identifies only two key reference points, the most important criterion (Best) and the least important criterion (Worst). Then, all other remaining criteria are compared against these two reference points. This method brings down the total number of pairwise comparisons required to just $2n-3$. Consequently, BWM significantly decreases the number of pairwise comparisons, saves time, reduces cognitive load, and provides consistent results \cite{rezaei2015best}.

The standard BWM is effective in determining the weights of criteria and is also beneficial for determining the priorities of alternatives \cite{corrente2024better}. However, in large-scale situations, a practical challenge emerges when the number of alternatives is larger than nine \cite{corrente2024better}. At this scale, even the simplified pairwise comparisons of BWM represent a major cognitive load for the DM. To address this practical challenge, the concept of Parsimonious BWM (P-BWM) was introduced in \cite{moslem2023novel} and subsequently formalized into a robust framework in \cite{corrente2024better}.
 This method does not require the DM to compare each alternative with the best and worst reference alternatives but rather to directly assign initial ratings to all of them. From this rated list, a small, well-distributed set of reference alternatives is selected. Then, the usual pairwise comparisons are made with respect to these few reference alternatives, applying the standard BWM to obtain their priorities. The priorities of all non-reference alternatives are derived using an interpolation formula. The cognitive load of the DM is drastically reduced by restricting the numerous pairwise comparisons to a few references. As a result, this approach saves significant time and effort while leading to highly consistent and reliable results \cite{corrente2024better}.

In the standard BWM context, the DM's judgments are naturally expressed by crisp numbers. However, in practice, DMs often find it difficult to convey the precise numerical values of their preferences due to the inherent vagueness of human judgment and the complexity of the criteria involved. Consequently, the standard BWM approach was modified to accommodate this uncertainty. Guo and Zhao \cite{guo2017fuzzy} were the first to extend BWM to a fuzzy environment, proposing the Fuzzy Best-Worst Method (FBWM). The pairwise comparisons were conducted using Triangular Fuzzy Numbers (TFNs), and the final weights of the criteria were derived using a non-linear optimization model.

\subsection{Research Gaps and Motivation}
However, the initial fuzzy extension of BWM, FBWM, had some limitations because it relied on approximate fuzzy arithmetic. Consequently, the model failed to capture the complete geometric shape of the TFNs provided by the DM, leading to a significant loss of important judgment information \cite{ratandhara2024alpha}. Moreover, the FBWM framework was based on a non-linear model that occasionally led to multiple optimal solutions, generating unwanted ambiguity in the final weights' determination.

To systematically address these challenges, Ratandhara and Kumar \cite{ratandhara2024alpha} developed an $\alpha$-cut intervals based FBWM ($\alpha$-FBWM). This improved method is able to maintain the entire geometric shape of the fuzzy pairwise comparison values by exact fuzzy arithmetic operations specified through $\alpha$-cut intervals. Additionally, the $\alpha$-FBWM framework resolves the issue of multiple optimal solutions. However, despite its precision, the $\alpha$-FBWM framework presents critical research gaps:
\begin{itemize}
   \item \textbf{Computational Complexity:} The existing $\alpha$-FBWM requires solving $2n + 1$ non-linear optimization models, followed by subsequent interval boundary averaging to determine a unique weight set for $n$  criteria. This approach is computationally intensive and time-consuming.

    \item \textbf{Loss of Fuzzy Uncertainty:} The unique weight set obtained from the existing $\alpha$-FBWM consists of crisp numbers. Consequently, the framework fails to capture the underlying uncertainty inherent in this unique weight set.
\end{itemize}   
\par
\noindent \textbf{Limitations of Existing Parsimonious Models of FBWM:} Recently, researchers have focused on developing parsimonious extensions of the FBWM to accommodate large-scale decision problems. For example, the Fuzzy Parsimonious Z-BWM \cite{fuzzyparsimoniousZBWM} was introduced in 2025 to evaluate commuter travel options in Dublin, while another parsimonious approach was applied in 2026 to rate Global Sustainable Development Goals (SDGs) \cite{fuzzyPBWMSDGs}. However, these existing extensions have two major drawbacks. First, they force experts to use crisp numbers for their initial ratings, which is very difficult when dealing with imprecise, real-world situations. Second, when determining the priorities of non-reference alternatives, these methods prematurely defuzzify the priorities of reference alternatives into crisp values prior to applying the interpolation formula.

These research gaps motivate the development of the proposed Linear $\alpha$-FBWM and Linear $\alpha$-PFBWM frameworks. This study aims to mitigate computational complexity, preserve uncertainty, and reduce the cognitive load of DM in large-scale applications.

\subsection{Key Contributions and Paper Organization}
To overcome these limitations, this study develops effective frameworks. The core contributions are outlined as follows:
\begin{itemize}
   \item \textbf{Linear Model (Linear $\alpha$-FBWM):} We propose a linear programming model that yields the unique weight set for criteria. Additionally, this unique weight set is obtained directly as Triangular Fuzzy Numbers (TFNs) instead of crisp numbers, which captures the uncertainties in the resulting weights.

\item \textbf{A Novel Consistency Metric (OPV):} To evaluate the ordinal consistency of the generated weights using FBWM, a unique metric named the Ordinal Preference Violation (OPV) is introduced based on the cognitive psychology-driven prominence effect. This metric assesses the alignment between the final weight rankings and the DM's initial fuzzy pairwise preferences established through the Best-to-Others (BO) and Others-to-Worst (OW) vectors. By applying an asymmetric evaluation, the OPV metric accurately measures any discrepancy between the model's results and the DM's actual ordinal intent.

 \item \textbf{A New Parsimonious Fuzzy Best-Worst Method via Linear $\alpha$-FBWM Integration (Linear $\alpha$-PFBWM):} To overcome the limitations of existing parsimonious frameworks of FBWM, we propose a new model called the Linear $\alpha$-cut intervals-based Parsimonious Fuzzy Best-Worst Method (Linear $\alpha$-PFBWM). By directly embedding our proposed Linear $\alpha$-FBWM formulation into its core structure, this framework enables DMs to provide initial ratings of alternatives directly as fuzzy numbers. Furthermore, the priorities of non-reference alternatives are computed using fuzzy interpolation formula, avoiding early defuzzification of the priorities of reference alternatives.
 
  \item \textbf{Real-World Large-Scale Application:} To demonstrate its practical utility, the proposed Linear $\alpha$-PFBWM is applied to a strategic decision-making problem: selecting the ideal warehouse locations from 20 potential sites across Gujarat for a leading multinational paint manufacturing company. By implementing the Linear $\alpha$-PFBWM, the DM only needed to provide 20 initial ratings to these locations and perform just 35 fuzzy pairwise comparisons, instead of the 185 fuzzy pairwise comparisons. This achieves a massive 81.08\% reduction in intensive fuzzy pairwise comparisons, significantly reducing the DM's cognitive load.
\end{itemize}

The rest of the paper is organized as follows:
\begin{itemize}
    \item \textbf{Section 2:} Presents the preliminaries covering fuzzy sets and their arithmetic operations and provides a brief review of $\alpha$-FBWM and parsimonious BWM.

    \item \textbf{Section 3:} Details the proposed mathematical framework, which contains the main contributions of this study: the Linear $\alpha$-FBWM, the uniqueness of the approximated optimal weight set (Theorem \ref{theorem_unique_opt_sol}), its Error of Approximation (EA) bound (Theorem \ref{theorem_EA}), the Ordinal Preference Violation (OPV) metric, and the Linear $\alpha$-PFBWM framework.

    \item \textbf{Section 4:} Presents the comparative analysis and numerical validation, where we compare the proposed Linear $\alpha$-FBWM with the existing $\alpha$-FBWM using examples and use a numerical example to show how to apply the proposed Linear $\alpha$-PFBWM.

   \item \textbf{Section 5:} Presents the industrial application for warehouse location selection using the Linear $\alpha$-PFBWM.

    \item \textbf{Section 6:} Concludes the work, summarizes the principal findings, and presents the limitations and future research directions.

    \item \textbf{Appendix A:} Contains the data collection questionnaire utilized during the expert elicitation process for the industrial warehouse location selection problem.
\end{itemize}

\theoremstyle{definition}
\newtheorem*{definition}{Definition}
\section{Preliminaries}
\subsection{Fuzzy Sets and Arithmetic Operations}

\begin{definition}[\cite{definationsoffuzzyzimmermann2011fuzzy}]
A fuzzy set $\tilde{A}$ over a universal set $X$ is defined by its membership function $\mu_{\tilde{A}}: X \to [0, 1]$, which maps each element $x \in X$ to a real value in the unit interval. Formally, it is expressed as the set of ordered pairs $\tilde{A} = \{(x, \mu_{\tilde{A}}(x)) \mid x \in X\}$, with $\mu_{\tilde{A}}(x)$ indicating the membership degree of $x$.
\end{definition}

\begin{definition}[\cite{definationsoffuzzyzimmermann2011fuzzy}]
 The support of a fuzzy set $\tilde{A}$, denoted by $S(\tilde{A})$, is defined as the crisp subset of the universal set $X$ containing all elements whose membership degree is strictly greater than zero. Formally, $S(\tilde{A}) = \{x \in X \mid \mu_{\tilde{A}}(x) > 0\}$.
\end{definition}

\begin{definition}[\cite{definationsoffuzzyzimmermann2011fuzzy}]
 For any $\alpha \in [0, 1]$, the $\alpha$-cut of a fuzzy set $\tilde{A}$, denoted as $\tilde{A}_\alpha$, is a crisp set comprising all elements in $X$ that have a membership degree greater than or equal to $\alpha$. It is formulated as $\tilde{A}_\alpha = \{x \in X \mid \mu_{\tilde{A}}(x) \geq \alpha\}$.
\end{definition}

\begin{definition}[\cite{definitionoffuzzynumberklir1996fuzzy}]
A fuzzy set $\tilde{A}$ defined on the real line $\mathbb{R}$ is considered a fuzzy number if it satisfies the following three conditions:
\begin{enumerate}
    \item There exists an $x \in \mathbb{R}$ for which $\mu_{\tilde{A}}(x) = 1$;
    \item Every $\alpha$-cut $\tilde{A}_\alpha$ is a closed interval for any $\alpha \in (0, 1]$;
    \item The support $S(\tilde{A})$ is a bounded set.
\end{enumerate}
\end{definition}

\begin{definition}[\cite{ratandhara2024alpha}]
A Triangular Fuzzy Number (TFN) $\tilde{A}$ is a special class of fuzzy numbers parameterized by a triplet $(l, m, u)$, where $l \leq m \leq u$. Its membership function $\mu_{\tilde{A}}(x)$ is formally defined as 
\begin{equation}\label{TFN}
\mu_{\tilde{A}}(x) = 
\begin{cases} 
\frac{x - l}{m - l}, & \text{if } l \leq x \leq m \\
\frac{u - x}{u - m}, & \text{if } m \leq x \leq u \\
0, & \text{otherwise}.
\end{cases}
\end{equation}
The graphical representation of the TFN is illustrated in Figure \ref{fig:tfn}. Furthermore, the $\alpha$-cut interval of the TFN $\tilde{A} = (l, m, u)$ can be explicitly derived by equating the membership function to the threshold $\alpha$, which results in the closed interval $\tilde{A}_\alpha = [a^l_{\alpha}, a^u_{\alpha}] = [l + \alpha(m - l), u - \alpha(u - m)]$.
\end{definition}

\begin{figure}[htbp]
    \centering
    \begin{tikzpicture}[scale=1.2, thick] 
        \draw[->] (-0.5,0) -- (6,0) node[right, font=\footnotesize] {$x$};
        \draw[->] (0,-0.2) -- (0,1.5) node[above, font=\footnotesize] {$\mu_{\tilde{A}}(x)$};
        \draw (-0.1,1) -- (0.1,1) node[left=5pt, font=\footnotesize] {$1$};
        \draw (-0.1,0.6) -- (0.1,0.6) node[left=5pt, font=\footnotesize] {$\alpha$};
        \draw[dashed, gray] (0,1) -- (3,1);
        \draw[blue, very thick] (1.5,0) -- (3,1) -- (4.5,0); 
        \node[blue, above, font=\footnotesize] at (3,1.1) {$\tilde{A} = (l, m, u)$}; 
        \draw[dashed, gray] (3,1) -- (3,0);
        \node[below=2pt, font=\footnotesize] at (1.5,0) {$l$};
        \node[below=2pt, font=\footnotesize] at (3,0) {$m$};
        \node[below=2pt, font=\footnotesize] at (4.5,0) {$u$};
        \draw[dashed, red, thick] (0,0.6) -- (2.4,0.6);
        \draw[red, very thick] (2.4,0.6) -- (3.6,0.6); 
        \draw[dashed, red, thick] (3.6,0.6) -- (5.2,0.6);
        \filldraw[red] (2.4,0.6) circle (1.5pt);
        \filldraw[red] (3.6,0.6) circle (1.5pt);
        \draw[dashed, red] (2.4,0.6) -- (2.4,0);
        \draw[dashed, red] (3.6,0.6) -- (3.6,0);
        \node[below=1pt, red, font=\scriptsize] at (2.4,0) {$a^l_{\alpha}$};
        \node[below=1pt, red, font=\scriptsize] at (3.6,0) {$a^u_{\alpha}$};
        \draw[very thick, red, |-|] (2.4,-0.6) -- (3.6,-0.6) node[midway, below, font=\footnotesize] {$\tilde{A}_\alpha$};
        
    \end{tikzpicture}
   \caption{Graphical representation of a Triangular Fuzzy Number $\tilde{A}$ and its corresponding $\alpha$-cut interval.}

    \label{fig:tfn}
\end{figure}
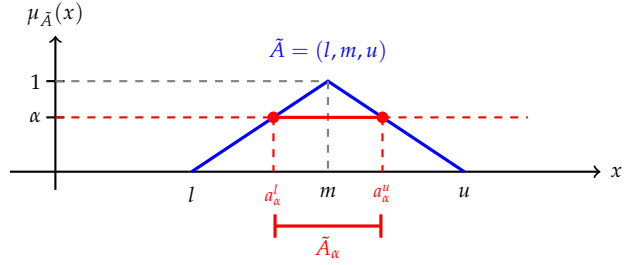
\begin{definition}[\cite{definationsoffuzzyzimmermann2011fuzzy}]
Let $\tilde{A}$ and $\tilde{B}$ be two strictly positive fuzzy numbers defined over the non-negative real line $\mathbb{R}^+$. For any given membership degree $\alpha \in (0, 1]$, let their respective closed $\alpha$-cut intervals be denoted as $\tilde{A}_\alpha = [a^l_\alpha, a^u_\alpha]$ and $\tilde{B}_\alpha = [b^l_\alpha, b^u_\alpha]$. Furthermore, let $k$ represent any real-valued scalar. The standard interval based algebraic operations between these $\alpha$-cuts are established as follows:
\begin{itemize}
    \item \textbf{Interval Addition:} 
    \begin{equation}\label{addition}
    (\tilde{A} \oplus \tilde{B})_\alpha = [a^l_\alpha + b^l_\alpha, a^u_\alpha + b^u_\alpha]
    \end{equation}
    
    \item \textbf{Interval Subtraction:} 
    \begin{equation}\label{subtraction}
    (\tilde{A} \ominus \tilde{B})_\alpha = [a^l_\alpha - b^u_\alpha, a^u_\alpha - b^l_\alpha]
    \end{equation}
    
    \item \textbf{Interval Multiplication:} 
    \begin{equation}\label{multiplication}
    (\tilde{A} \otimes \tilde{B})_\alpha = [a^l_\alpha \times b^l_\alpha, a^u_\alpha \times b^u_\alpha]
    \end{equation}
    
    \item \textbf{Interval Division:} 
    \begin{equation}\label{Division}
    (\tilde{A} \oslash \tilde{B})_\alpha = \left[ \frac{a^l_\alpha}{b^u_\alpha}, \frac{a^u_\alpha}{b^l_\alpha} \right]
    \end{equation}
    
    \item \textbf{Scalar Multiplication:} The transformation by a scalar $k$ strictly depends on its algebraic sign:
    \begin{equation}\label{scalar_multiplication}
    k \otimes \tilde{A}_\alpha = 
    \begin{cases} 
      [k a^l_\alpha, k a^u_\alpha], & \text{if } k > 0 \\
      [k a^u_\alpha, k a^l_\alpha], & \text{if } k < 0 
    \end{cases}
    \end{equation}
\end{itemize}
\end{definition}

\begin{definition}[\cite{fuzzynumberarithchen1985operations}]
Let $\tilde{A} = (a_1, a_2, a_3)$ and $\tilde{B} = (b_1, b_2, b_3)$ be two TFNs, and let $k$ be a real-valued scalar. The arithmetic operations for these fuzzy numbers are given by:
\begin{itemize}
    \item \textbf{Fuzzy Addition:}
    \begin{equation}\label{F_addition}
    \tilde{A} \oplus \tilde{B} = (a_1 + b_1, a_2 + b_2, a_3 + b_3)
    \end{equation}
    
    \item \textbf{Fuzzy Subtraction:}
    \begin{equation}\label{F_subtraction}
    \tilde{A} \ominus \tilde{B} = (a_1 - b_3, a_2 - b_2, a_3 - b_1)
    \end{equation}
    
    \item \textbf{Fuzzy Multiplication:} 
    \begin{equation}\label{F_Multiplication}
    \tilde{A} \otimes \tilde{B} = \left( \min(a_1 b_1, a_1 b_3, a_3 b_1, a_3 b_3), a_2 b_2, \max(a_1 b_1, a_1 b_3, a_3 b_1, a_3 b_3) \right)
    \end{equation}
    
    \item \textbf{Fuzzy Division:}
    \begin{equation}\label{F_division}
    \tilde{A} \oslash \tilde{B} = \left( \min \left( \frac{a_1}{b_1}, \frac{a_1}{b_3}, \frac{a_3}{b_1}, \frac{a_3}{b_3} \right), \frac{a_2}{b_2}, \max \left( \frac{a_1}{b_1}, \frac{a_1}{b_3}, \frac{a_3}{b_1}, \frac{a_3}{b_3} \right) \right)
    \end{equation}

    \item \textbf{Scalar Multiplication:}
    \begin{equation}\label{T_Scalar_multiplication}
    k \otimes \tilde{A} = 
    \begin{cases} 
      (k a_1, k a_2, k a_3), & \text{if } k \geq 0 \\
      (k a_3, k a_2, k a_1), & \text{if } k < 0 
    \end{cases}
    \end{equation}
\end{itemize}
\end{definition}

\begin{definition}[\cite{guo2017fuzzy}]
For a given Triangular Fuzzy Number $\tilde{A} = (a, b, c)$, its Graded Mean Integration Representation (GMIR), denoted by $R(\tilde{A})$, is evaluated as follows:
\begin{equation}\label{GMIR}
R(\tilde{A}) = \frac{a + 4b + c}{6}
\end{equation}
\end{definition}
\begin{remark}\label{lineartyofR}
Based on the algebraic structure of Eq. (\ref{GMIR}), the GMIR operator $R$ naturally satisfies the property of linearity. For any two TFNs $\tilde{A} = (a_1, b_1, c_1)$ and $\tilde{B} = (a_2, b_2, c_2)$, and a real scalar $k$, the following properties hold:
\begin{enumerate}
    \item \textbf{Additivity:} $R(\tilde{A} + \tilde{B}) = R(\tilde{A}) + R(\tilde{B})$
    \item \textbf{Homogeneity:} $R(k \cdot \tilde{A}) = k \cdot R(\tilde{A})$
\end{enumerate}
\end{remark}

\noindent\textbf{Lipschitz Continuity} (\cite{rudin1976principles}) Let $g: I \rightarrow \mathbb{R}$ be a real-valued function defined on an interval $I \subseteq \mathbb{R}$. The function $g$ is said to be Lipschitz continuous on $I$ if there exists a real constant $M \ge 0$ such that
$$|g(x) - g(y)| \le M|x - y|$$
for all $x, y \in I$. The smallest such constant $M$ is referred to as the Lipschitz constant for the function $g$.

\vspace{0.2cm}

\noindent \textbf{Mean Value Theorem (MVT)}(\cite{rudin1976principles}) Let $f$ be a real-valued function that is continuous on the closed interval $[a, b]$ and differentiable on the open interval $(a, b)$. Then there exists a point $x_0 \in (a, b)$ such that $f'(x_0) = \frac{f(b) - f(a)}{b - a}$.

\subsection{A brief review of the \texorpdfstring{$\alpha$}{alpha}-cut intervals based FBWM (\textbf{\texorpdfstring{$\alpha$}{alpha}-FBWM})} \label{brief_review_alpha-cut intervals FBWM}

 The $\alpha$-cut intervals based FBWM ($\alpha$-FBWM), introduced by Ratandhara and Kumar \cite{ratandhara2024alpha}, provides a systematic framework for deriving a unique weight set for decision criteria. The  core elements and optimization models are defined as follows:

 Let $C=\{c_1, c_2, \dots, c_n\}$ be the set of decision criteria, and let $c_b$ and $c_w$ denote the best and worst criteria, respectively. The Best-to-Other vector 
    $\widetilde{A}_b = (\tilde{a}_{b1}, \tilde{a}_{b2}, \dots, \tilde{a}_{bn})$
and the Other-to-Worst vector 
    $\widetilde{A}_w = (\tilde{a}_{1w}, \tilde{a}_{2w}, \dots, \tilde{a}_{nw})$
are together called the Fuzzy Pairwise Comparison System (FPCS). This is denoted by 
    $(\widetilde{A}_b, \widetilde{A}_w)$. The values of $(\widetilde{A}_b, \widetilde{A}_w)$ are determined according to the linguistic scale presented in Table \ref{tab:linguistic_tfns}.

    \begin{table}[htbp]
    \centering
    \caption{Linguistic terms and associated TFNs \cite{ratandhara2024alpha}.}
    \label{tab:linguistic_tfns}
    \renewcommand{\arraystretch}{1.2}
    \begin{tabular}{ll}
        \hline
        Linguistic term & TFN \\
        \hline
        Equally preference & $\tilde{1} = (1, 1, 1)$ \\
        Weakly preference & $\tilde{3} = (2, 3, 4)$ \\
        Essentially preference & $\tilde{5} = (4, 5, 6)$ \\
        Very strong preference & $\tilde{7} = (6, 7, 8)$ \\
        Absolutely preference & $\tilde{9} = (9, 9, 9)$ \\
        \hline
        \multirow{4}{*}{Intermediate values} & $\tilde{2} = (1, 2, 3)$ \\
        & $\tilde{4} = (3, 4, 5)$ \\
        & $\tilde{6} = (5, 6, 7)$ \\
        & $\tilde{8} = (7, 8, 9)$ \\
        \hline
    \end{tabular}
\end{table}

\vspace{0.2cm}

Let $\widetilde{w}_i = (w_i^l, w_i^m, w_i^u)$ be any Triangular Fuzzy Number (TFN) representing the fuzzy weight of the $i$-th criterion, and let $\tilde{a}_{ij} = (a_{ij}^l, a_{ij}^m, a_{ij}^u)$ denote the relative fuzzy preference of the $i$-th criterion over the $j$-th criterion. Then, for any membership degree $\alpha \in (0, 1]$, their respective $\alpha$-cut intervals are defined as follows:
\begin{align}
(\widetilde{w}_i)_\alpha &= [w_i^l + \alpha(w_i^m - w_i^l), \, w_i^u - \alpha(w_i^u - w_i^m)], \\
(\tilde{a}_{ij})_\alpha &= [a_{ij}^l(\alpha), \, a_{ij}^u(\alpha)] \nonumber \\
&= [a_{ij}^l + \alpha(a_{ij}^m - a_{ij}^l), \, a_{ij}^u - \alpha(a_{ij}^u - a_{ij}^m)],
\end{align}
and for $\alpha = 0$,
\begin{align}
(\widetilde{w}_i)_0 &= [w_i^l, w_i^u], \\
(\tilde{a}_{ij})_0 &= [a_{ij}^l(0), \, a_{ij}^u(0)] = [a_{ij}^l, a_{ij}^u].
\end{align}

\begin{definition}[\cite{ratandhara2024alpha}]
    A weight set $\{\widetilde{w}_1, \widetilde{w}_2, \dots, \widetilde{w}_n\}$ is said to be accurate if it satisfies the following equation
\begin{equation}\label{accurate_weight_set}
\begin{aligned}
\frac{w_b^l + \alpha(w_b^m - w_b^l)}{w_i^u - \alpha(w_i^u - w_i^m)} &= a_{bi}^l(\alpha), & \frac{w_b^u - \alpha(w_b^u - w_b^m)}{w_i^l + \alpha(w_i^m - w_i^l)} &= a_{bi}^u(\alpha), \\[1em]
\frac{w_i^l + \alpha(w_i^m - w_i^l)}{w_w^u - \alpha(w_w^u - w_w^m)} &= a_{iw}^l(\alpha), & \frac{w_i^u - \alpha(w_i^u - w_i^m)}{w_w^l + \alpha(w_w^m - w_w^l)} &= a_{iw}^u(\alpha), \\[1em]
\frac{w_b^l + \alpha(w_b^m - w_b^l)}{w_w^u - \alpha(w_w^u - w_w^m)} &= a_{bw}^l(\alpha), & \frac{w_b^u - \alpha(w_b^u - w_b^m)}{w_w^l + \alpha(w_w^m - w_w^l)} &= a_{bw}^u(\alpha)
\end{aligned}
\end{equation}
  \end{definition}
 for all $\alpha \in [0, 1]$ and $i \in \{1, 2, \dots, n\} \setminus \{b, w\}$.

 \vspace{0.2cm}
 
It is possible that the system of equations (\ref{accurate_weight_set}) may not yield a feasible solution, implying that an accurate weight set does not exist. In such cases, an optimal weight set is determined by solving the following non-linear optimization model:
\begin{equation}\label{nonlinearoptimization}
\begin{aligned}
& \min \epsilon^t \\
& \text{subject to:} \\
& \left| \frac{w_b^l + \alpha(w_b^m - w_b^l)}{w_i^u - \alpha(w_i^u - w_i^m)} - a_{bi}^l(\alpha) \right| \leq \epsilon^t, \qquad 
  \left| \frac{w_b^u - \alpha(w_b^u - w_b^m)}{w_i^l + \alpha(w_i^m - w_i^l)} - a_{bi}^u(\alpha) \right| \leq \epsilon^t, \\[1ex]
& \left| \frac{w_i^l + \alpha(w_i^m - w_i^l)}{w_w^u - \alpha(w_w^u - w_w^m)} - a_{iw}^l(\alpha) \right| \leq \epsilon^t, \qquad
  \left| \frac{w_i^u - \alpha(w_i^u - w_i^m)}{w_w^l + \alpha(w_w^m - w_w^l)} - a_{iw}^u(\alpha) \right| \leq \epsilon^t, \\[1ex]
& \left| \frac{w_b^l + \alpha(w_b^m - w_b^l)}{w_w^u - \alpha(w_w^u - w_w^m)} - a_{bw}^l(\alpha) \right| \leq \epsilon^t, \qquad
  \left| \frac{w_b^u - \alpha(w_b^u - w_b^m)}{w_w^l + \alpha(w_w^m - w_w^l)} - a_{bw}^u(\alpha) \right| \leq \epsilon^t, \\[1ex]
& 0 \leq w_q^l \leq w_q^m \leq w_q^u, \quad \text{(well-definedness and non-negativity of weights)} \\[1ex]
& \sum_{i=1}^n R(\tilde{w}_i) = 1 \quad \text{(normalization)} \\[1ex]
& \text{for } i \in \{1, 2, \dots, n\} \setminus \{b, w\}, \alpha \in [0, 1] \text{ and } q \in \{1, 2, \dots, n\}.
\end{aligned}
\end{equation}

\vspace{0.2cm}

The optimal solution for problem (\ref{nonlinearoptimization}) is represented by $\tilde{w}_i^* = (w_i^{l*}, w_i^{m*}, w_i^{u*})$ and $\epsilon^{t*}$. Here, $\{\tilde{w}_1^*, \tilde{w}_2^*, \dots, \tilde{w}_n^*\}$ constitutes the optimal fuzzy weight set and $\epsilon^{t*}$ indicates the accuracy of this derived weight set.
As can be observed, problem (\ref{nonlinearoptimization}) involves an infinite number of constraints, making it computationally difficult to solve directly. To overcome this difficulty, a finite subset $F = \{0 = \alpha_1, \alpha_2, \dots, \alpha_k = 1\} \subseteq [0,1]$ is considered, such that $\alpha_1 < \alpha_2 < \dots < \alpha_k$. The mesh of $F$ is defined as $\|F\|_\infty = \max\{\alpha_{i+1} - \alpha_i \mid i = 1, 2, \dots, k - 1\}$. Subsequently, the approximate optimal weight set corresponding to the subset $F$ can be determined by solving the following non-linear optimization model:

\begin{equation}\label{F_nonlinearopt}
\begin{aligned}
& \min \epsilon_F^t \\
& \text{subject to:} \\
& \left| \frac{w_b^l + \alpha(w_b^m - w_b^l)}{w_i^u - \alpha(w_i^u - w_i^m)} - a_{bi}^l(\alpha) \right| \leq \epsilon_F^t, \qquad 
  \left| \frac{w_b^u - \alpha(w_b^u - w_b^m)}{w_i^l + \alpha(w_i^m - w_i^l)} - a_{bi}^u(\alpha) \right| \leq \epsilon_F^t, \\[1ex]
& \left| \frac{w_i^l + \alpha(w_i^m - w_i^l)}{w_w^u - \alpha(w_w^u - w_w^m)} - a_{iw}^l(\alpha) \right| \leq \epsilon_F^t, \qquad
  \left| \frac{w_i^u - \alpha(w_i^u - w_i^m)}{w_w^l + \alpha(w_w^m - w_w^l)} - a_{iw}^u(\alpha) \right| \leq \epsilon_F^t, \\[1ex]
& \left| \frac{w_b^l + \alpha(w_b^m - w_b^l)}{w_w^u - \alpha(w_w^u - w_w^m)} - a_{bw}^l(\alpha) \right| \leq \epsilon_F^t, \qquad
  \left| \frac{w_b^u - \alpha(w_b^u - w_b^m)}{w_w^l + \alpha(w_w^m - w_w^l)} - a_{bw}^u(\alpha) \right| \leq \epsilon_F^t, \\[1ex]
& 0 \leq w_q^l \leq w_q^m \leq w_q^u, \quad \text{(well-definedness and non-negativity of weights)} \\[1ex]
& \sum_{i=1}^n R(\tilde{w}_i) = 1 \quad \text{(normalization)} \\[1ex]
& \text{for } i \in \{1, 2, \dots, n\} \setminus \{b, w\}, \alpha \in F \text{ and } q \in \{1, 2, \dots, n\}.
\end{aligned}
\end{equation}

\noindent After solving model (\ref{F_nonlinearopt}), the set $\{\tilde{w}_{1F}^*, \tilde{w}_{2F}^*, \dots, \tilde{w}_{nF}^*\}$ is obtained, which represents the approximate optimal fuzzy weight set corresponding to $F$, while $\epsilon_F^{t*}$ serves as a measure of its accuracy. Consequently, the set $\{R(\tilde{w}_{1F}^*), R(\tilde{w}_{2F}^*), \dots, R(\tilde{w}_{nF}^*)\}$ forms the corresponding defuzzified approximate optimal weight set, where $R(\tilde{w}_{iF}^*)$ denotes the GMIR of $\tilde{w}_{iF}^*$. However, since model (\ref{F_nonlinearopt}) remains a non-linear optimization problem, it may yield multiple approximate optimal weight sets. To address this multiplicity, the optimization models (\ref{GLB_opt_model}) and (\ref{LUB_opt_model}) are solved to determine the Greatest Lower Bound (GLB) and Least Upper Bound (LUB) for the optimal weight of each criterion. These bounds establish an interval, and the average value of this interval is considered the final crisp weight for the respective criterion.

\begin{equation}\label{GLB_opt_model}
\begin{aligned}
& \min R(\tilde{w}_p) =\left( \frac{w_p^l + 4w_p^m + w_p^u}{6} \right) \\
& \text{subject to:} \\
& \left| \frac{w_b^l + \alpha(w_b^m - w_b^l)}{w_i^u - \alpha(w_i^u - w_i^m)} - a_{bi}^l(\alpha) \right| \leq \epsilon_F^{t*}, \qquad 
  \left| \frac{w_b^u - \alpha(w_b^u - w_b^m)}{w_i^l + \alpha(w_i^m - w_i^l)} - a_{bi}^u(\alpha) \right| \leq \epsilon_F^{t*}, \\[1ex]
& \left| \frac{w_i^l + \alpha(w_i^m - w_i^l)}{w_w^u - \alpha(w_w^u - w_w^m)} - a_{iw}^l(\alpha) \right| \leq \epsilon_F^{t*}, \qquad
  \left| \frac{w_i^u - \alpha(w_i^u - w_i^m)}{w_w^l + \alpha(w_w^m - w_w^l)} - a_{iw}^u(\alpha) \right| \leq \epsilon_F^{t*}, \\[1ex]
& \left| \frac{w_b^l + \alpha(w_b^m - w_b^l)}{w_w^u - \alpha(w_w^u - w_w^m)} - a_{bw}^l(\alpha) \right| \leq \epsilon_F^{t*}, \qquad
  \left| \frac{w_b^u - \alpha(w_b^u - w_b^m)}{w_w^l + \alpha(w_w^m - w_w^l)} - a_{bw}^u(\alpha) \right| \leq \epsilon_F^{t*}, \\[1ex]
& 0 \leq w_q^l \leq w_q^m \leq w_q^u, \\[1ex]
& \sum_{i=1}^n R(\tilde{w}_i) = 1 \\[1ex]
& \text{for } i \in \{1, 2, \dots, n\} \setminus \{b, w\}, \alpha \in F \text{ and } q \in \{1, 2, \dots, n\} \text{ and}
\end{aligned}
\end{equation}

\begin{equation}\label{LUB_opt_model}
\begin{aligned}
& \max R(\tilde{w}_p)= \left( \frac{w_p^l + 4w_p^m + w_p^u}{6} \right) \\
& \text{subject to:} \\
& \left| \frac{w_b^l + \alpha(w_b^m - w_b^l)}{w_i^u - \alpha(w_i^u - w_i^m)} - a_{bi}^l(\alpha) \right| \leq \epsilon_F^{t*}, \qquad 
  \left| \frac{w_b^u - \alpha(w_b^u - w_b^m)}{w_i^l + \alpha(w_i^m - w_i^l)} - a_{bi}^u(\alpha) \right| \leq \epsilon_F^{t*}, \\[1ex]
& \left| \frac{w_i^l + \alpha(w_i^m - w_i^l)}{w_w^u - \alpha(w_w^u - w_w^m)} - a_{iw}^l(\alpha) \right| \leq \epsilon_F^{t*}, \qquad
  \left| \frac{w_i^u - \alpha(w_i^u - w_i^m)}{w_w^l + \alpha(w_w^m - w_w^l)} - a_{iw}^u(\alpha) \right| \leq \epsilon_F^{t*}, \\[1ex]
& \left| \frac{w_b^l + \alpha(w_b^m - w_b^l)}{w_w^u - \alpha(w_w^u - w_w^m)} - a_{bw}^l(\alpha) \right| \leq \epsilon_F^{t*}, \qquad
  \left| \frac{w_b^u - \alpha(w_b^u - w_b^m)}{w_w^l + \alpha(w_w^m - w_w^l)} - a_{bw}^u(\alpha) \right| \leq \epsilon_F^{t*}, \\[1ex]
& 0 \leq w_q^l \leq w_q^m \leq w_q^u, \\[1ex]
& \sum_{i=1}^n R(\tilde{w}_i) = 1 \\[1ex]
& \text{for } i \in \{1, 2, \dots, n\} \setminus \{b, w\}, \alpha \in F \text{ and } q \in \{1, 2, \dots, n\},
\end{aligned}
\end{equation}
where $\epsilon_F^*$ is the optimal objective value corresponding to $F$ in model (\ref{F_nonlinearopt}).

\subsection{A brief review of Parsimonious BWM (P-BWM)}

The Parsimonious BWM (P-BWM), refined by Corrente et al. \cite{corrente2024better}, is utilized to efficiently evaluate the priorities of various alternatives while minimizing the cognitive burden of pairwise comparisons. In scenarios involving a large number of alternatives, the systematic procedure for calculating their priorities is outlined in the following steps:

\textbf{Step 1:} Let $A = \{a_1, a_2, \dots, a_n\}$ denote the set of all available alternatives. In this initial step, the DM assigns a specific rating $r(a_j) \in \mathbb{R}$ to each alternative $a_j \in A$. The assigned rating $r(a_j)$ represents the relative goodness or overall performance of alternative $a_j$ in comparison to the others.

\textbf{Step 2:} In this step, the DM selects a predefined set of $t$ reference alternatives, denoted as $(a_{\gamma_1}, a_{\gamma_2}, \dots, a_{\gamma_t})$, from the main set of alternatives $A$. This selection process is formulated and executed by solving the following Mixed Integer Linear Programming (MILP) \cite{corrente2024better} model:
\begin{equation}\label{MILP}
\begin{aligned}
& \max \varepsilon = \varepsilon^*, \quad \text{subject to} \\[1ex]
& \left.
\begin{aligned}
& \rho_j \cdot r(a_j) - \rho_{j'} \cdot r(a_{j'}) \geqslant \varepsilon - (2 - \rho_j - \rho_{j'}) \cdot M \\
& \quad \text{for all } j, j' = 1, \dots, n, \text{ such that } r(a_j) \geqslant r(a_{j'}) \\[1ex]
& \sum_{j=1}^n \rho_j = t, \\
& \rho_{(1)} = 1, \rho_{(n)} = 1, \\
& \rho_{(j)} + \rho_{(j+1)} + \rho_{(j+2)} \leqslant 2, \quad \text{for all } j = 1, \dots, n - 2, \\
& \rho_j \in \{0, 1\} \quad \text{for all } j = 1, \dots, n.
\end{aligned}
\right\}
\end{aligned}
\end{equation}

\textbf{Step 3:} The DM, assisted by the analyst, applies the BWM to the established set of reference alternatives $\{a_{\gamma_1}, \dots, a_{\gamma_t}\}$. Initially, the DM identifies the best alternative $a_B$ and the worst alternative $a_W$ within this subset. Subsequently, pairwise comparisons are conducted between $a_B$, $a_W$, and the remaining reference alternatives to construct the Best-to-Other (BO) and Other-to-Worst (OW) vectors. If the Global Input-based Consistency Ratio ($CR^I$) remains within the acceptable threshold, the priorities of the reference alternatives, denoted as $u(r(a_{\gamma_1})), \dots, u(r(a_{\gamma_t}))$, are computed. Conversely, if the consistency threshold is exceeded, the DM must revise the provided preference information to resolve the inconsistency before reapplying the BWM to derive the priorities. 

\textbf{Step 4:} In this step, the analyst determines the priorities of the non-reference alternatives. For each alternative $a_j \in A$ whose rating falls within the interval $r(a_j) \in [r(a_{\gamma_s}), r(a_{\gamma_{s+1}})]$, its corresponding priority $w_j = u(r(a_j))$ is computed by linearly interpolating the priorities $u(r(a_{\gamma_s}))$ and $u(r(a_{\gamma_{s+1}}))$ obtained in the previous step. This is calculated using the following formula:

\begin{equation}\label{linear_interpolation_formula}
w_j = u(r(a_j)) = u(r(a_{\gamma_s})) + \frac{u(r(a_{\gamma_{s+1}})) - u(r(a_{\gamma_s}))}{r(a_{\gamma_{s+1}}) - r(a_{\gamma_s})} \left( r(a_j) - r(a_{\gamma_s}) \right).
\end{equation}

\section{Proposed Mathematical Framework}
\subsection{Linear Formulation for \texorpdfstring{$\alpha$}{alpha}-cut Intervals based FBWM (Linear \texorpdfstring{$\alpha$}{alpha}-FBWM)}
As established in the previous section \ref{brief_review_alpha-cut intervals FBWM}, model (\ref{nonlinearoptimization}) is inherently non-linear in nature. One of the major drawbacks of such non-linear optimization frameworks is their tendency to yield multiple optimal weight sets \cite{ratandhara2024alpha}. This creates unwanted ambiguity, making it difficult for the DM to finalize a single, reliable optimal weight set. To systematically resolve this critical issue, we reformulate the existing framework as a linear programming model, as presented below:
\begin{equation}\label{linearformulationoptimization}
\begin{aligned}
    & \min \eta^t \\
    & \text{Subject to:} \\
    & \left| w_b^l + \alpha(w_b^m - w_b^l) - a_{bi}^l(\alpha) \cdot \left( w_i^u - \alpha(w_i^u - w_i^m) \right) \right| \le \eta^t \\
    & \left| w_i^l + \alpha(w_i^m - w_i^l) - a_{iw}^l(\alpha) \cdot \left( w_w^u - \alpha(w_w^u - w_w^m) \right) \right| \le \eta^t \\
    & \left| w_b^l + \alpha(w_b^m - w_b^l) - a_{bw}^l(\alpha) \cdot \left( w_w^u - \alpha(w_w^u - w_w^m) \right) \right| \le \eta^t \\
    & \left| w_b^u - \alpha(w_b^u - w_b^m) - a_{bi}^u(\alpha) \cdot \left( w_i^l + \alpha(w_i^m - w_i^l) \right) \right| \le \eta^t \\
    & \left| w_i^u - \alpha(w_i^u - w_i^m) - a_{iw}^u(\alpha) \cdot \left( w_w^l + \alpha(w_w^m - w_w^l) \right) \right| \le \eta^t \\
    & \left| w_b^u - \alpha(w_b^u - w_b^m) - a_{bw}^u(\alpha) \cdot \left( w_w^l + \alpha(w_w^m - w_w^l) \right) \right| \le \eta^t \\
    & 0 \le w_q^l \le w_q^m \le w_q^u, \\
    & \sum_{i=1}^n R(\tilde{w}_i) = 1 \\
    & \text{for } i \in \{1, 2, \dots, n\} \setminus \{b, w\}, \quad \alpha \in [0, 1] \quad \text{and} \quad q \in \{1, 2, \dots, n\}.
\end{aligned}
\end{equation}

Here, the optimal solution of model (\ref{linearformulationoptimization}) is given by $\tilde{w}^*_i = (w^{l*}_i, w^{m*}_i, w^{u*}_i)$ and $\eta^{t*}$, where $\{\tilde{w}^*_1, \tilde{w}^*_2, \dots, \tilde{w}^*_n\}$ is the optimal triangular fuzzy weight set (referred to as the optimal weight set in this article), while the value of $\eta^{t*}$ serves as a direct indicator of the accuracy and consistency of this optimal weight set.

\subsection{Uniqueness and Error of Approximation Bound of the Optimal Weight Set}

However, a significant computational challenge arises when attempting to solve model (\ref{linearformulationoptimization}) directly. Since $\alpha$ can take any value within the interval $[0, 1]$, the optimization model contains an infinite number of constraints, making it intractable in its original form. 
To systematically address this intractability, the interval $[0, 1]$ is divided into a finite partition $F = \{\alpha_1, \alpha_2, \dots, \alpha_k\}$, such that $0 = \alpha_1 < \alpha_2 < \dots < \alpha_k = 1$.
 
Let $\|F\|_\infty = \max \{\alpha_{i+1} - \alpha_i \mid i = 1, 2, \dots, k - 1\}$, where $\|F\|_\infty$ is called the mesh of $F$. By employing this partition $F$, the infinite constraints are reduced to a finite number of workable boundaries. Consequently, the approximate optimal weight set corresponding to the partition $F$ can be determined by solving the following linear programming model:
\begin{equation}\label{F_linearformulation}
\begin{aligned}
    & \min \eta_F^t \\
    & \text{Subject to:} \\
    & \left| w_b^l + \alpha \cdot (w_b^m - w_b^l) - a_{bi}^l(\alpha) \cdot \left( w_i^u - \alpha(w_i^u - w_i^m) \right) \right| \le \eta_F^t \\
    & \left| w_i^l + \alpha \cdot (w_i^m - w_i^l) - a_{iw}^l(\alpha) \cdot \left( w_w^u - \alpha(w_w^u - w_w^m) \right) \right| \le \eta_F^t \\
    & \left| w_b^l + \alpha \cdot (w_b^m - w_b^l) - a_{bw}^l(\alpha) \cdot \left( w_w^u - \alpha(w_w^u - w_w^m) \right) \right| \le \eta_F^t \\
    & \left| w_b^u - \alpha\cdot(w_b^u - w_b^m) - a_{bi}^u(\alpha) \cdot \left( w_i^l + \alpha(w_i^m - w_i^l) \right) \right| \le \eta_F^t \\
    & \left| w_i^u - \alpha\cdot(w_i^u - w_i^m) - a_{iw}^u(\alpha) \cdot \left( w_w^l + \alpha(w_w^m - w_w^l) \right) \right| \le \eta_F^t \\
    & \left| w_b^u - \alpha\cdot(w_b^u - w_b^m) - a_{bw}^u(\alpha) \cdot \left( w_w^l + \alpha(w_w^m - w_w^l) \right) \right| \le \eta_F^t \\
    & 0 \le w_q^l \le w_q^m \le w_q^u, \\
    & \sum_{i=1}^n R(\tilde{w}_i) = 1 \\
    & \text{for } i \in \{1, 2, \dots, n\} \setminus \{b, w\}, \quad \alpha \in F \quad \text{and} \quad q \in \{1, 2, \dots, n\}.
\end{aligned}
\end{equation}

Note that model (\ref{F_linearformulation}) is a minimization problem with $3n + 1$ variables. Solving this model yields an optimal solution denoted by $(\tilde{w}_{1F}^*, \tilde{w}_{2F}^*, \dots, \tilde{w}_{nF}^*,  \eta_F^{t^*})$. Within this solution, the set $\{\tilde{w}_{1F}^*, \tilde{w}_{2F}^*, \dots, \tilde{w}_{nF}^*\}$ represents the approximate optimal weight set derived for the specified partition $F$, while the optimal value $\eta_{F}^{t*}$ serves as a direct indicator of its accuracy and consistency. For notational simplicity, the subscript $F$ may be omitted in subsequent sections when the choice of partition is unambiguous.

\begin{theorem}\label{theorem_unique_opt_sol}
For any given partition $F = \{\alpha_1, \alpha_2, \dots, \alpha_k\} \subset [0, 1]$ such that $0 = \alpha_1 < \alpha_2 < \dots < \alpha_k = 1$, the linear programming model (\ref{F_linearformulation}) yields a unique optimal solution.
\end{theorem}

\begin{proof}
To establish this theorem, the fundamental properties of linear programming are utilized. The proof is structured into three distinct phases:

\begin{enumerate}
    \item The feasible region is non-empty: The objective function of model (\ref{F_linearformulation}) aims to minimize the error variable $\eta_F^t$. In the formulated constraints, $\eta_F^t$ serves as an upper bound for the absolute deviations. For any arbitrary set of fuzzy weights that satisfy both the well-definedness conditions and the normalization constraint, one can always select a sufficiently large positive value for $\eta_F^t$ such that all absolute deviation inequality constraints are simultaneously satisfied. Since at least one such feasible point invariably exists, the feasible region is non-empty.

    \item The feasible region is bounded: The variables representing the fuzzy weights are strictly bounded within a compact space. First, they are constrained from below by zero due to the non-negativity and well-definedness conditions, $0 \le w_q^l \le w_q^m \le w_q^u$. Second, they are upper-bounded because the sum of their GMIR values must equal unity, as dictated by the normalization constraint $\sum_{i=1}^n R(\tilde{w}_i) = 1$. As all variables are between zero and a finite upper bound, the feasible region forms a closed and bounded polytope.

    \item The optimal solution is unique: According to the fundamental theorem of linear programming \cite{linearprogrammingmodel}, a linear program defined over a non-empty, bounded feasible region attains its optimum at an extreme point (vertex) of the polytope. In model (\ref{F_linearformulation}), the objective function is to minimize the single variable $\eta_F^t$. The gradient of this objective hyperplane is linearly independent of the gradients of the bounding constraint hyperplanes. Consequently, the objective hyperplane cannot be parallel to any face or edge of the feasible region, eliminating the possibility of an infinite number of solutions. It intersects the polytope at exactly one unique vertex, guaranteeing a single, unique optimal solution.
\end{enumerate}
This completes the proof.
\end{proof}

Next, we establish the Error of Approximation (EA) bound for the approximate optimal weight set derived from model (\ref{F_linearformulation}).

\begin{theorem}\label{theorem_EA}
Let $(\tilde{A}_b, \tilde{A}_w)$ be an FPCS having the comparison values from the scale given in Table \ref{tab:linguistic_tfns}. Let $F = \{\alpha_1, \alpha_2, \dots, \alpha_k\} \subset [0, 1]$ such that 
$0 = \alpha_1 < \alpha_2 < \dots < \alpha_k = 1.$
Let $\eta^{t^*}$ be the optimal objective value of the model (\ref{linearformulationoptimization}) and $\eta_F^{t^*}$ be the optimal objective value of the model (\ref{F_linearformulation}) corresponding to $F$. Then there exists a positive constant $M$ such that $|\eta^{t^*} - \eta_F^{t^*}| \leq M \|F\|_{\infty}.$
\end{theorem}

\begin{proof}
    Since $F \subset [0, 1]$, the model (\ref{F_linearformulation}) possesses a subset of the constraints found in the model (\ref{linearformulationoptimization}). Consequently, it holds that $\eta_F^{t^*} \leq \eta^{t^*}$. 
Let $\{ \tilde{w}_{1F}^*, \tilde{w}_{2F}^*, \dots, \tilde{w}_{nF}^* \}$ represent the approximate optimal weight set computed with respect to the partition $F$. By definition, for any $\alpha_j \in F$, all inequality constraints in model (\ref{F_linearformulation}) are satisfied and bounded by the optimal error $\eta_F^{t^*}$.

\noindent So we have,
\begin{align*}
\max \Big\{ &| (w_{bF}^{l*} + \alpha_j(w_{bF}^{m*} - w_{bF}^{l*})) - a_{bi}^{l}(\alpha_j) \cdot (w_{iF}^{u*} - \alpha_j(w_{iF}^{u*} - w_{iF}^{m*})) |, \\
&| (w_{bF}^{u*} - \alpha_j(w_{bF}^{u*} - w_{bF}^{m*})) - a_{bi}^{u}(\alpha_j) \cdot (w_{iF}^{l*} + \alpha_j(w_{iF}^{m*} - w_{iF}^{l*})) |, \\
&| (w_{iF}^{l*} + \alpha_j(w_{iF}^{m*} - w_{iF}^{l*})) - a_{iw}^{l}(\alpha_j) \cdot (w_{wF}^{u*} - \alpha_j(w_{wF}^{u*} - w_{wF}^{m*})) |, \\
&| (w_{iF}^{u*} - \alpha_j(w_{iF}^{u*} - w_{iF}^{m*})) - a_{iw}^{u}(\alpha_j) \cdot (w_{wF}^{l*} + \alpha_j(w_{wF}^{m*} - w_{wF}^{l*})) |, \\
&| (w_{bF}^{l*} + \alpha_j(w_{bF}^{m*} - w_{bF}^{l*})) - a_{bw}^{l}(\alpha_j) \cdot (w_{wF}^{u*} - \alpha_j(w_{wF}^{u*} - w_{wF}^{m*})) |, \\
&| (w_{bF}^{u*} - \alpha_j(w_{bF}^{u*} - w_{bF}^{m*})) - a_{bw}^{u}(\alpha_j) \cdot (w_{wF}^{l*} + \alpha_j(w_{wF}^{m*} - w_{wF}^{l*})) | \\
&: i = 1, 2, \dots, n, i \neq b, i \neq w, \alpha_j \in F \Big\} = \eta_F^{t^*}
\end{align*}

$\leq \eta_F^{t^*} + M \|F\|_{\infty} \leq \eta^{t^*} + M \|F\|_{\infty}, \text{ for any } M > 0.$  
Hence, for all $\alpha_j \in F$, the absolute deviations do not exceed $\eta^{t^*} + M \|F\|_{\infty}$. 
\par
Now, consider any $\alpha \in [0, 1] \setminus F$. There exist two consecutive partition points $\alpha_j, \alpha_{j+1} \in F$ such that $\alpha_j < \alpha < \alpha_{j+1}$.   
Let $g_1(\alpha) = (w_{bF}^{l*} + \alpha(w_{bF}^{m*} - w_{bF}^{l*})) - a_{bi}^{l}(\alpha) \cdot (w_{iF}^{u*} - \alpha(w_{iF}^{u*} - w_{iF}^{m*}))$. Since $(w_{bF}^{l*} + \alpha(w_{bF}^{m*} - w_{bF}^{l*}))$, $(w_{iF}^{u*} - \alpha(w_{iF}^{u*} - w_{iF}^{m*}))$ and $a_{bi}^{l}(\alpha)$ are all linear functions of $\alpha \in [0, 1]$, the function $g_1(\alpha)$ is a quadratic polynomial in $\alpha$. Therefore $g_1(\alpha)$ is continuous and differentiable on $[0, 1]$. By the mean value theorem, for any $\alpha_j$ and $\alpha$ there exists a point $c \in (\alpha_j, \alpha)$ such that $|g_1(\alpha) - g_1(\alpha_j)| = |g'_1(c)| \cdot |\alpha - \alpha_j|$. Since the domain $[0, 1]$ is a closed and bounded interval, and $g'_1(\alpha)$ is a continuous linear function, the absolute value of the derivative is bounded by some finite constant $M_1$; that is $|g'_1(\alpha)| \leq M_1$ for all $\alpha \in [0, 1]$. This shows that $g_1(\alpha)$ is Lipschitz continuous with constant $M_1$, therefore $|g_1(\alpha) - g_1(\alpha_j)| \leq M_1 \cdot |\alpha - \alpha_j| \leq M_1 \cdot |\alpha_{j+1} - \alpha_j|$.
Since $\alpha_j < \alpha < \alpha_{j+1}$ the distance $|\alpha_{j+1} - \alpha_j|$ is strictly less than the mesh size $\|F\|_{\infty}$. Thus  $|g_1(\alpha) - g_1(\alpha_j)| \leq M_1 \|F\|_{\infty}$. Using the triangle inequality, $|g_1(\alpha)| - |g_1(\alpha_j)| \leq |g_1(\alpha) - g_1(\alpha_j)| \leq M_1 \|F\|_{\infty}$, therefore
$|g_1(\alpha)| \leq |g_1(\alpha_j)| + M_1 \|F\|_{\infty}$. Since $|g_1(\alpha_j)| \leq \eta_F^{t^*}$, we have $|(w_{bF}^{l*} + \alpha(w_{bF}^{m*} - w_{bF}^{l*})) - a_{bi}^{l}(\alpha) \cdot (w_{iF}^{u*} - \alpha(w_{iF}^{u*} - w_{iF}^{m*}))| \leq \eta_F^{t^*} + M_1 \|F\|_{\infty}$. By applying the same Lipschitz continuity argument to the five remaining absolute deviation constraints, we can determine finite bounding constants $M_2, M_3, \dots, M_6$. Defining $M = \max\{M_1, M_2, \dots, M_6\}$, the maximum deviation across all constraints for any $\alpha \in [0, 1]$ does not exceed $\eta_F^{t*} + M \|F\|_{\infty}$.
This implies that the optimal weight set from the model (\ref{F_linearformulation}) is a feasible solution to the model (\ref{linearformulationoptimization}) with an error of at most $\eta_F^{t^*} + M\|F\|_{\infty}$. Hence, the optimal error of the model (\ref{linearformulationoptimization}), $\eta^{t^*}$ must satisfy
\[
\eta^{t^*} \leq \eta_F^{t^*} + M\|F\|_{\infty}
\]
\[
\therefore \quad 0 \leq \eta^{t^*} - \eta_F^{t^*} \leq M\|F\|_{\infty}
\]
\[
\therefore \quad |\eta^{t^*} - \eta_F^{t^*}| \leq M\|F\|_{\infty}.
\]
This completes the proof.
\end{proof}

\begin{remark}\label{Remark_upper_bound}
The existence of a finite constant $M$ in Theorem \ref{theorem_EA} is crucial as it ensures that the approximation error between the model (\ref{linearformulationoptimization}) and model (\ref{F_linearformulation}) remains bounded. Consequently, we can establish a strict upper bound for this constant based on the inherent properties of the linguistic fuzzy pairwise comparison scale in Table~\ref{tab:linguistic_tfns}. 

Since, $$g_1(\alpha) = (w_{bF}^{l*} + \alpha(w_{bF}^{m*} - w_{bF}^{l*})) - a_{bi}^l(\alpha) \cdot (w_{iF}^{u*} - \alpha(w_{iF}^{u*} - w_{iF}^{m*})).$$
Differentiating $g_1(\alpha)$ with respect to $\alpha$, and applying the triangle inequality, the absolute rate of change is upper-bounded by:

$$|g_1'(\alpha)| \le (w_{bF}^{m*} - w_{bF}^{l*}) + (a_{bi}^m - a_{bi}^l) \cdot(w_{iF}^{u*} - \alpha(w_{iF}^{u*} - w_{iF}^{m*}) + a_{bi}^l(\alpha) \cdot (w_{iF}^{u*} - w_{iF}^{m*}).$$

To determine the absolute upper bound over the interval $\alpha \in [0, 1]$, we analyze the monotonic behavior of the functions involved. Since $(w_{iF}^{u*} - \alpha(w_{iF}^{u*} - w_{iF}^{m*})$ is a decreasing function of $\alpha$, its maximum value occurs at $\alpha = 0$, yielding $w_{iF}^{u*}$. Conversely, $a_{bi}^l(\alpha)$ is an increasing function, reaching its maximum at $\alpha = 1$, which gives $a_{bi}^l(1) = a_{bi}^m$. Substituting these maximum values into the equation above gives the exact bound:
$$|g_1'(\alpha)| \le (w_{bF}^{m*} - w_{bF}^{l*}) + (a_{bi}^m - a_{bi}^l) \cdot w_{iF}^{u*} + a_{bi}^m \cdot (w_{iF}^{u*} - w_{iF}^{m*}).$$

For any normalized fuzzy weight, the absolute maximum value and its spread are bounded by $1$. Furthermore, based on the standard 1-9 linguistic scale in Table~\ref{tab:linguistic_tfns} utilized in this study, the maximum spread of any TFN is $1$ (e.g., $a_{bi}^m - a_{bi}^l \le 1$), and the maximum value of $a_{bi}^m$ is $9$. Substituting these extreme values into the inequality yields:
$$|g_1'(\alpha)| \le 1 + (1)(1) + 9(1) = 11.$$

By applying the same argument and symmetry to the remaining constraint functions $g_k(\alpha)$ for $k = 2, 3, 4, 5, 6$, it can be systematically proven that they are all subject to the same upper limit. Thus, a universal finite upper bound $M \le 11$ exists. This bound theoretically confirms that the approximation error $|\eta^{t*} - \eta_F^{t*}|$ is bounded by $11 \|F\|_\infty$.
\end{remark}

\begin{remark}
\label{rem:natural_extension}
As the traditional $\alpha$-FBWM was a natural extension of the non-linear BWM \cite{ratandhara2024alpha}, our proposed Linear $\alpha$-FBWM is a natural extension of the linear model of BWM.
\end{remark}

\subsection{Novel Cognitive-Based Ordinal Preference Violation (OPV) Metric}

One of the most important properties of an MCDM method is to preserve the ordinal preferences of the DM \cite{ordinalviolation}. To check how well the final derived defuzzified weights from any FBWM maintain the initial fuzzy preference vectors (Best-to-Others and Others-to-Worst), we introduce a novel metric called the Ordinal Preference Violation (OPV). This metric provides a reliable way to systematically measure this specific structural consistency by tracking logical contradictions between computed rankings and initial inputs.

In decision-making psychology, the prominence effect \cite{Prominenceeffect} represents a well-documented phenomenon. According to this concept, during the evaluation of multi-attribute alternatives, cognitive attention naturally shifts toward the most prominent or important attribute. Consequently, this dominant dimension receives the highest focus throughout the decision-making process. Conversely, criteria of lesser importance receive significantly less mental focus, which tends to amplify uncertainty and decision conflict during their evaluation. Because DMs typically exhibit higher confidence and precision when comparing criteria against the most important (Best) criterion, contradicting the Best-to-Others ($\tilde{A}_B$) vector constitutes a more severe structural violation than contradicting the Others-to-Worst ($\tilde{A}_W$) vector. 

To model this asymmetric cognitive behavior, the proposed framework evaluates separate localized discrepancy scores for the $\tilde{A}_B$ and $\tilde{A}_W$ preference orders. These discrepancy scores are primarily triggered when a strict ranking order exists between the final derived criteria weights (i.e., $w_j > w_i$). Additionally, to ensure a comprehensive evaluation, the metric assesses scenarios where the derived weights are equal (i.e., $w_j = w_i$) yet conflict with the initial preferences. Under these conditions, the original fuzzy vectors are inspected for potential structural reversals.

First, the individual discrepancy score for violating the Best-to-Others vector ($V_{ij}^B$) is formulated as follows:
\begin{equation}\label{BO_discrepancy}
    V_{ij}^B = 
    \begin{cases} 
    1, & \text{if } w_j > w_i \text{ and } R(\tilde{a}_{Bj}) > R(\tilde{a}_{Bi}) \\
    0.5, & \text{if } w_j > w_i \text{ and } R(\tilde{a}_{Bj}) = R(\tilde{a}_{Bi}) \\
    0.5, & \text{if } w_j = w_i, R(\tilde{a}_{Bj}) \neq R(\tilde{a}_{Bi}), \text{ and } j > i \\
    0, & \text{otherwise}
    \end{cases}
\end{equation}

Here, $R(\cdot)$ denotes the GMIR value of the corresponding TFN. A discrepancy score of $1$ represents a full ordinal contradiction, whereas $0.5$ indicates a partial contradiction. Such partial contradictions occur either when a tie exists in the defuzzified initial preferences despite strictly ordered final weights, or when the final weights are tied despite strictly ordered GMIR values of the initial preferences. To avoid double-counting symmetric criteria pairs in the equal-weight case, the constraint $j > i$ is used.

Similarly, the individual discrepancy score for violating the Others-to-Worst vector ($V_{ij}^W$) is formulated as follows:
\begin{equation}\label{OW_discrepancy}
    V_{ij}^W = 
    \begin{cases} 
    0.5, & \text{if } w_j > w_i \text{ and } R(\tilde{a}_{jW}) < R(\tilde{a}_{iW}) \\
    0.25, & \text{if } w_j > w_i \text{ and } R(\tilde{a}_{jW}) = R(\tilde{a}_{iW}) \\
    0.25, & \text{if } w_j = w_i, R(\tilde{a}_{jW}) \neq R(\tilde{a}_{iW}), \text{ and } j > i \\
    0, & \text{otherwise}
    \end{cases}
\end{equation}
In this context, the maximum individual discrepancy score is scaled down to 0.5 to reflect the lower cognitive clarity and higher evaluation uncertainty inherently associated with the $\tilde{A}_W$ vector variations. 

Finally, by utilizing these localized discrepancy components, the overall combined metric is constructed. To ensure that the metric remains scale-independent, the sum of all individual violations is normalized by its theoretical maximum upper bound. For a problem with $n$ criteria, the maximum possible aggregated discrepancy is derived as $\frac{3n(n-1)}{4}$. Thus, the final Ordinal Preference Violation (OPV) metric is evaluated using the following formula:
\begin{equation}\label{OPV}
    OPV = \frac{4}{3n(n-1)} \sum_{i=1}^{n} \sum_{j=1}^{n} \left( V_{ij}^B + V_{ij}^W \right)
\end{equation}
By incorporating this normalization factor, the aggregated OPV score lies within the interval $[0, 1]$, where an OPV of $0$ indicates flawless ordinal consistency and a score of $1$ signifies a complete preference reversal relative to the DM's initial inputs.

\subsection{Proposed Linear \texorpdfstring{$\alpha$}{alpha}-cut intervals based Parsimonious FBWM with fuzzy rating (Linear \texorpdfstring{$\alpha$}{alpha}-PFBWM)}

In the Fuzzy Best-Worst Method (FBWM), when the number of alternatives becomes very large, it becomes very difficult for Decision-Makers (DMs) to provide accurate fuzzy pairwise comparison vectors. This creates a heavy cognitive load on the DMs, which can easily lead to significant inconsistency in the final priorities. To solve this problem and ensure that DMs experience a much lower cognitive load, this section proposes a new and improved framework named the Linear $\alpha$-cut intervals-based Parsimonious Fuzzy Best-Worst Method (Linear $\alpha$-PFBWM). By following the systematic steps described below, we can determine the priorities for all alternatives while keeping the DM's evaluation burden minimal.

\paragraph{\textbf{Step 1: Assigning Initial Fuzzy Ratings}}
Let $A = \{a_1, a_2, \dots, a_n\}$ denote the set of available alternatives to be evaluated. In this initial step, the DM assigns a Triangular Fuzzy Number (TFN), denoted by $\tilde{r}(a_j) = (l_j, m_j, u_j)$, to represent the fuzzy performance rating of each alternative $a_j \in A$. This assigned rating reflects the relative preference or overall performance of alternative $a_j$ against the remaining alternatives.

\paragraph{\textbf{Step 2: Selection of Well-Distributed Reference Alternatives}}
In this step, a predefined set of $t$ reference alternatives, denoted as $(a_{\gamma_1}, a_{\gamma_2}, \dots, a_{\gamma_t})$, is selected from the main set of alternatives $A$. To ensure that these selected reference alternatives are well-distributed across the entire range of alternative sizes, the analyst solves the Mixed-Integer Linear Programming (MILP) \cite{corrente2024better} optimization framework shown in Model (\ref{PFBWM_MILP}) for the DM. In this MILP formulation, $m_j$ represents the middle value of the initial Triangular Fuzzy Number (TFN) rating assigned to alternative $a_j$ in the previous step (Step 1). By running this optimization model, the analyst extracts the best representative reference alternatives and provides them to the DM:
\begin{equation}\label{PFBWM_MILP}
\begin{aligned}
& \max \varepsilon = \varepsilon^*, \quad \text{subject to} \\[1ex]
& \left.
\begin{aligned}
& \rho_j \cdot m_j - \rho_{j'} \cdot m_{j'} \geqslant \varepsilon - (2 - \rho_j - \rho_{j'}) \cdot M \\
& \quad \text{for all } j, j' = 1, \dots, n, \text{ such that } m_j \geqslant m_{j'} \\[1ex]
& \sum_{j=1}^n \rho_j = t, \\
& \rho_{(1)} = 1, \rho_{(n)} = 1, \\
& \rho_{(j)} + \rho_{(j+1)} + \rho_{(j+2)} \leqslant 2, \quad \text{for all } j = 1, \dots, n - 2, \\
& \rho_j \in \{0, 1\} \quad \text{for all } j = 1, \dots, n.
\end{aligned}
\right\}
\end{aligned}
\end{equation}

\paragraph{\textbf{Step 3: Determining Priorities of Reference Alternatives}}
The objective of this step is to determine the priorities of the selected reference alternatives $\{a_{\gamma_1}, \dots, a_{\gamma_t}\}$. Initially, the DM identifies the best alternative $a_B$ and the worst alternative $a_W$ within this subset. Subsequently, fuzzy pairwise comparisons are executed based on the linguistic scale provided in Table~\ref{tab:linguistic_tfns}, yielding the Best-to-Other (BO) and Other-to-Worst (OW) fuzzy comparison vectors.
 
Before computing the priorities, the consistency of the DM's judgments must be verified. Accordingly, the Global Input-based Consistency Ratio ($CR^I$) \cite{Inputbasedconsistency} is calculated using the formulations defined in Eq.~(\ref{CR}) and Eq.~(\ref{CR_j}). If $CR^I$ remains within the acceptable threshold provided in Table~\ref{tab:consistency_thresholds}, the evaluation process can reliably proceed. Conversely, if the consistency threshold is exceeded, the DM must adjust the preference information to resolve contradictions prior to recalculation. Finally, the proposed Linear $\alpha$-FBWM optimization model from Section~3.1 is applied to determine the priorities of the reference alternatives, denoted as $\tilde{u}(\tilde{r}(a_{\gamma_1})), \dots, \tilde{u}(\tilde{r}(a_{\gamma_t}))$.

\begin{equation}\label{CR}
CR^I = \max_j CR_j^I
\end{equation}
where
\begin{equation}\label{CR_j}
CR_j^I = 
\begin{cases} 
\left| \frac{R(\tilde{a}_{Bj} \otimes \tilde{a}_{jW} - \tilde{a}_{BW})}{R(\tilde{a}_{BW} \otimes \tilde{a}_{BW} - \tilde{a}_{BW})} \right| & \tilde{a}_{BW} \neq (1, 1, 1) \\[2ex]
0 & \tilde{a}_{BW} = (1, 1, 1)
\end{cases}
\end{equation}

\begin{table}[htbp]
\centering
\caption{Thresholds for different combinations using input-based consistency measurement \cite{Inputbasedconsistency}.}
\label{tab:consistency_thresholds}
\begin{tabular}{lccccccc}
\toprule
\multirow{2}{*}{Scales} & \multicolumn{7}{c}{Criteria} \\ \cmidrule(l){2-8} 
 & 3 & 4 & 5 & 6 & 7 & 8 & 9 \\ \midrule
3 & 0.1667 & 0.1667 & 0.1667 & 0.1667 & 0.1667 & 0.1667 & 0.1667 \\
4 & 0.1121 & 0.1529 & 0.1898 & 0.2206 & 0.2527 & 0.2577 & 0.2683 \\
5 & 0.1354 & 0.1994 & 0.2306 & 0.2546 & 0.2716 & 0.2844 & 0.2960 \\
6 & 0.1330 & 0.1990 & 0.2643 & 0.3044 & 0.3144 & 0.3221 & 0.3262 \\
7 & 0.1294 & 0.2457 & 0.2819 & 0.3029 & 0.3144 & 0.3251 & 0.3403 \\
8 & 0.1309 & 0.2521 & 0.2958 & 0.3154 & 0.3408 & 0.3620 & 0.3657 \\
9 & 0.1359 & 0.2681 & 0.3062 & 0.3337 & 0.3517 & 0.3620 & 0.3662 \\ \bottomrule
\end{tabular}
\end{table}

\paragraph{\textbf{Step 4: Determining Priorities of Non-Reference Alternatives}}
In the final step, the priorities of all non-reference alternatives are determined. For each alternative $a_j \in A$ whose initial fuzzy rating's middle value lies within the specific interval $m_j \in [m_{\gamma_s}, m_{\gamma_{s+1}}]$, its corresponding crisp priority $R(\tilde{w}_j) = R(\tilde{u}(\tilde{r}(a_j)))$ is directly computed. This is achieved by linearly interpolating the reference priorities $\tilde{u}(\tilde{r}(a_{\gamma_s}))$ and $\tilde{u}(\tilde{r}(a_{\gamma_{s+1}}))$ obtained in Step~3. Using the linearity of the GMIR operator $R$ (\ref{lineartyofR}), this interpolation is formulated as:

\begin{equation*}
R(\tilde{w}_j) = R(\tilde{u}(\tilde{r}(a_j))) = R\left[ \tilde{u}(\tilde{r}(a_{\gamma_s})) + \frac{\tilde{u}(\tilde{r}(a_{\gamma_{s+1}})) - \tilde{u}(\tilde{r}(a_{\gamma_s}))}{\tilde{r}(a_{\gamma_{s+1}}) - \tilde{r}(a_{\gamma_s})} \left( \tilde{r}(a_j) - \tilde{r}(a_{\gamma_s}) \right) \right]
\end{equation*}
\begin{equation}\label{fuzzy_interpolation}
= R(\tilde{u}(\tilde{r}(a_{\gamma_s}))) + R\left(\frac{\tilde{u}(\tilde{r}(a_{\gamma_{s+1}})) - \tilde{u}(\tilde{r}(a_{\gamma_s}))}{\tilde{r}(a_{\gamma_{s+1}}) - \tilde{r}(a_{\gamma_s})} \left( \tilde{r}(a_j) - \tilde{r}(a_{\gamma_s}) \right)\right)
\end{equation}

\section{Comparative Analysis and Numerical Validation}
\subsection{Comparative Analysis of the Proposed Linear \texorpdfstring{$\alpha$}{alpha}-FBWM}
In this section, the practical applicability of the proposed Linear $\alpha$-FBWM is demonstrated, highlighting its superiority over existing methods. To begin, the essential notations are introduced. Let $F_k$ represent the subset of $[0, 1]$ containing $k$ elements that partition $[0, 1]$ into $k-1$ uniform sub-intervals, defined as follows:
\begin{equation}
\label{eq:fk_definition}
F_k = \left\{0, \frac{1}{k-1}, \frac{2}{k-1}, \dots, 1\right\}
\end{equation}
For example, setting $k = 2$ yields $F_2 = \{0, 1\}$, while $k = 11$ results in $F_{11} = \{0, 0.1, 0.2, \dots, 0.9, 1\}$. When applying the proposed $\alpha$-FBWM, the weights are computed by setting $k = 2^i + 1$, where $i \in \mathbb{N} \cup \{0\}$. This specific choice of $k$ ensures that each successive partition bisects every sub-interval of the preceding partition into two equal intervals.

For a given FPCS $(\tilde{A}_b, \tilde{A}_w)$, let $\eta^*$ be the optimal objective value of model (\ref{linearformulationoptimization}), and let $\eta^*_{F_k}$ be the optimal objective value of model (\ref{F_linearformulation}) formulated using $F_k$. Here $\|F\|_\infty = \frac{1}{k-1}$. According to Theorem~\ref{theorem_EA} and Remark~\ref{Remark_upper_bound}, the bound is established as $|\eta^* - \eta^*_{F_k}| \le \frac{11}{k-1}$. This guarantees that the Error of Approximation (EA) of an approximate weight set obtained using $F_k$ is at most $\frac{11}{k-1}$. Therefore, the value of $k$ can be chosen to obtain a weight set that satisfies the desired EA.
 To validate and compare the performance of the proposed Linear $\alpha$-FBWM against existing methods, two numerical examples from the MCDM literature \cite{ratandhara2024alpha} are solved below:

\vspace{0.3cm}

\noindent\textbf{Example 1:} Let $C = \{c_1, c_2, \dots, c_5\}$ be the set of decision criteria, where $c_2$ and $c_5$ are the Best and the Worst criteria, respectively. The fuzzy Best-to-Others and Others-to-Worst vectors provided by the DM are $\tilde{A}_b = (\tilde{2}, \tilde{1}, \tilde{4}, \tilde{2}, \tilde{8})$ and $\tilde{A}_w = (\tilde{3}, \tilde{8}, \tilde{5}, \tilde{4}, \tilde{1})^T$. 

\vspace{0.3cm}

\noindent\textbf{Example 2:}  Let $C = \{c_1, c_2, \dots, c_5\}$ be the set of decision criteria, where $c_2$ and $c_5$ are the Best and the Worst criteria, respectively. The fuzzy Best-to-Others and Others-to-Worst vectors provided by the DM are $\tilde{A}_b = (\tilde{3}, \tilde{1}, \tilde{3}, \tilde{2}, \tilde{6})$ and $\tilde{A}_w = (\tilde{2}, \tilde{6}, \tilde{6}, \tilde{3}, \tilde{1})^T$.

\paragraph{\textbf{Computational Efficiency and Uncertainty Analysis}}
For any specified value of $k$, existing $\alpha$-FBWM approaches require solving 11 distinct non-linear optimization models followed by subsequent interval boundary averaging to determine a unique weight set. This multi-stage process is computationally intensive and time-consuming. Furthermore, since the resulting unique criteria weights are generated as crisp numbers, the inherent fuzzy uncertainty within the weights is entirely lost.

Conversely, when solving these examples using the proposed Linear $\alpha$-FBWM, only one linear programming model is solved to determine the unique weight set. Moreover, the resulting weights are preserved as Triangular Fuzzy Numbers (TFNs), which capture the underlying uncertainty inherent within the criteria weights.

\paragraph{\textbf{Analysis of Computed Weights}}
By applying the proposed Linear $\alpha$-FBWM, the criteria weights are computed for $i = 0, 7,$ and $9$, corresponding to $k = 2, 129,$ and $513$, respectively. The detailed fuzzy weights and their corresponding defuzzified values for Example 1 and Example 2 are compiled in Table~\ref{tab:example1_computed_weights} and Table~\ref{tab:example2_computed_weights}.

An analysis of these tables demonstrates the computational behavior of the EA bound. For both numerical examples, as $k$ increases from $2$ to $513$, the maximum theoretical EA value drops from $11$ to $0.0214844$. This monotonic decrease confirms that selecting a larger partition size $k$ significantly minimizes approximation errors, yielding a precise and stable fuzzy representation of the criteria weights.

\paragraph{\textbf{Comparative Analysis and Consistency Verification}}
The derived defuzzified weights are compared with those obtained from the existing $\alpha$-FBWM in Table~\ref{tab:comparison_example1} for Example~1 and Table~\ref{tab:comparison_example2} for Example~2. Existing literature \cite{ratandhara2024alpha} indicates that the existing $\alpha$-FBWM outperforms existing methods such as F-SBWM \cite{F-SBWM} and GP-FBWM \cite{GP-FBWM}. However, the computational results demonstrate that the proposed Linear $\alpha$-FBWM achieves superior consistency and alignment compared to the existing $\alpha$-FBWM.

In the proposed Linear $\alpha$-FBWM framework, the optimal error variable $\eta^*_{F_k}$ serves as a direct indicator of consistency, where values closer to zero signify higher consistency. In Example 1 (Table~\ref{tab:comparison_example1}), the proposed framework yields an optimal error of $\eta^*_{F_k} = 0.0854$, compared to $1.3945$ in the existing model. This optimization error remains significantly lower than the upper bound of the Consistency Ratio (CR) of the existing $\alpha$-FBWM, which stands at $0.3138$. Similarly, for Example 2 (Table~\ref{tab:comparison_example2}), the proposed model achieves an optimal error of $\eta^*_{F_k} = 0.1095$, in contrast to $1.5360$ in the existing model, remaining well below the upper bound of the CR of the existing $\alpha$-FBWM, which stands at $0.5146$.

To verify the ordinal reliability of the generated weights, the Ordinal Preference Violation (OPV) is evaluated as a percentage. As shown in Table~\ref{tab:comparison_example1}, the proposed method exhibits a lower OPV of $10\%$ compared to $13.33\%$ in the existing method for Example 1, indicating superior alignment with the DM's original ordinal preferences. For Example 2 (Table~\ref{tab:comparison_example2}), both methods exhibit an identical low OPV of $11.67\%$. These metrics confirm that the Linear $\alpha$-FBWM reduces computational overhead while enhancing the consistency and ordinal reliability of the final weights.

To further validate the performance of the proposed Linear $\alpha$-FBWM via the OPV metric, eight additional numerical datasets (Examples 3 to 10) are examined. The Best-to-Others ($\tilde{A}_b$) and Others-to-Worst ($\tilde{A}_w$) fuzzy preference vectors for these cases are detailed in Table~\ref{tab:fuzzy_vectors}. Both methods were applied to these datasets to compute their respective OPV percentages.

The comparative results are illustrated in Figure~\ref{fig:opv_comparison}, where the red and blue markers represent the OPV values generated by the existing $\alpha$-FBWM and the proposed Linear $\alpha$-FBWM, respectively. This visual comparison demonstrates that the Linear $\alpha$-FBWM consistently yields OPV values that are lower than or equal to those of the existing $\alpha$-FBWM. This confirms that the proposed Linear $\alpha$-FBWM effectively preserves the DM's ordinal preferences while significantly reducing the computational effort.

\begin{remark}
\label{rem:k_selection}
Unless specified otherwise, the partition parameter is configured at $i = 7$ (corresponding to $k = 129$) for solving the proposed Linear $\alpha$-FBWM model throughout this study.
\end{remark}

\begin{table*}[htbp]
    \centering
    \caption{Computed weights for Example 1 using the proposed Linear $\alpha$-FBWM.}
    \label{tab:example1_computed_weights}
    \resizebox{\textwidth}{!}{
    \begin{tabular}{l c c c c c c}
        \toprule
        \textbf{Criterion} & \multicolumn{2}{c}{$k = 2$} & \multicolumn{2}{c}{$k = 129$} & \multicolumn{2}{c}{$k = 513$} \\
        \cmidrule(lr){2-3} \cmidrule(lr){4-5} \cmidrule(lr){6-7}
        & \textbf{Approx. optimal weight} & \textbf{Defuzzified weight} & \textbf{Approx. optimal weight} & \textbf{Defuzzified weight} & \textbf{Approx. optimal weight} & \textbf{Defuzzified weight} \\
        \midrule
        $c_1$ & $(0.1651, 0.2097, 0.2296)$ & $0.2056$ & $(0.1666, 0.2088, 0.2316)$ & $0.2056$ & $(0.1666, 0.2088, 0.2316)$ & $0.2056$ \\
        $c_2$ & $(0.3142, 0.4108, 0.4108)$ & $0.3947$ & $(0.3171, 0.4144, 0.4144)$ & $0.3982$ & $(0.3171, 0.4144, 0.4144)$ & $0.3982$ \\
        $c_3$ & $(0.0991, 0.1239, 0.1329)$ & $0.1212$ & $(0.0988, 0.1234, 0.1340)$ & $0.1211$ & $(0.0987, 0.1234, 0.1340)$ & $0.1211$ \\
        $c_4$ & $(0.1651, 0.2477, 0.2658)$ & $0.2370$ & $(0.1445, 0.2479, 0.2682)$ & $0.2340$ & $(0.1446, 0.2478, 0.2682)$ & $0.2340$ \\
        $c_5$ & $(0.0363, 0.0417, 0.0459)$ & $0.0415$ & $(0.0366, 0.0411, 0.0460)$ & $0.0412$ & $(0.0366, 0.0411, 0.0460)$ & $0.0412$ \\
        \midrule
        \textbf{EA} $\leq$ & \multicolumn{2}{c}{$11$} & \multicolumn{2}{c}{$0.0859375$} & \multicolumn{2}{c}{$0.0214844$} \\
        \bottomrule
    \end{tabular}
    }
\end{table*}

\begin{table}[htbp]
    \centering
    \caption{Comparison of the computed defuzzified weights and consistency for Example 1.}
    \label{tab:comparison_example1}
    \begin{tabular}{lcc}
        \toprule
        \textbf{Criterion} & \textbf{Proposed Linear $\alpha$-FBWM ($k=129$)} & \textbf{Existing $\alpha$-FBWM ($k=129$)} \\
        \midrule
        $c_1$ & $0.2056$ & $0.1671$ \\
        $c_2$ & $0.3982$ & $0.4367$ \\
        $c_3$ & $0.1211$ & $0.1676$ \\
        $c_4$ & $0.2340$ & $0.1884$ \\
        $c_5$ & $0.0412$ & $0.0465$ \\
        \midrule
        $\eta^*_{F_k}$ & $0.0854$ & $1.3945$ \\
        CR  & $-$ & $\leq$ $0.3138$ \\
        OPV (\% ) & 10 & 13.33\\
        \bottomrule
    \end{tabular}
\end{table}

\begin{table*}[htbp]
    \centering
    \caption{Computed weights for Example 2 using the proposed Linear $\alpha$-FBWM.}
    \label{tab:example2_computed_weights}
    \resizebox{\textwidth}{!}{
    \begin{tabular}{l c c c c c c}
        \toprule
        \textbf{Criterion} & \multicolumn{2}{c}{$k = 2$} & \multicolumn{2}{c}{$k = 129$} & \multicolumn{2}{c}{$k = 513$} \\
        \cmidrule(lr){2-3} \cmidrule(lr){4-5} \cmidrule(lr){6-7}
        & \textbf{Approx. optimal weight} & \textbf{Defuzzified weight} & \textbf{Approx. optimal weight} & \textbf{Defuzzified weight} & \textbf{Approx. optimal weight} & \textbf{Defuzzified weight} \\
        \midrule
        $c_1$ & $(0.1325, 0.1617, 0.2279)$ & $0.1678$ & $(0.1314, 0.1626, 0.2144)$ & $0.1660$ & $(0.1315, 0.1626, 0.2145)$ & $0.1660$ \\
        $c_2$ & $(0.3480, 0.3772, 0.4221)$ & $0.3799$ & $(0.3453, 0.3828, 0.4283)$ & $0.3841$ & $(0.3453, 0.3828, 0.4283)$ & $0.3841$ \\
        $c_3$ & $(0.1325, 0.1617, 0.2279)$ & $0.1678$ & $(0.1277, 0.1639, 0.2097)$ & $0.1655$ & $(0.1277, 0.1639, 0.2097)$ & $0.1655$ \\
        $c_4$ & $(0.1766, 0.2425, 0.2874)$ & $0.2390$ & $(0.1615, 0.2445, 0.2917)$ & $0.2385$ & $(0.1616, 0.2445, 0.2917)$ & $0.2385$ \\
        $c_5$ & $(0.0449, 0.0449, 0.0481)$ & $0.0454$ & $(0.0456, 0.0456, 0.0472)$ & $0.0458$ & $(0.0456, 0.0456, 0.0472)$ & $0.0458$ \\
        \midrule
        \textbf{EA} $\leq$ & \multicolumn{2}{c}{$11$} & \multicolumn{2}{c}{$0.0859375$} & \multicolumn{2}{c}{$0.0214844$} \\
        \bottomrule
    \end{tabular}
    }
\end{table*}

\begin{table}[htbp]
    \centering
    \caption{Comparison of the computed defuzzified weights and consistency for Example 2.}
    \label{tab:comparison_example2}
    \begin{tabular}{lcc}
        \toprule
        \textbf{Criterion} & \textbf{Proposed Linear $\alpha$-FBWM ($k=129$)} & \textbf{Existing $\alpha$-FBWM ($k=129$)} \\
        \midrule
        $c_1$ & $0.1660$ & $0.1394$ \\
        $c_2$ & $0.3841$ & $0.3922$ \\
        $c_3$ & $0.1655$ & $0.252 3$ \\
        $c_4$ & $0.2385$ & $0.1777$ \\
        $c_5$ & $0.0458$ & $0.0540$ \\
        \midrule
        $\eta^*_{F_k}$ & $0.1095$ & $1.5360$ \\
        CR & $-$ &  $\leq$ $0.5146$ \\
        OPV (\% ) & 11.67 & 11.67 \\
        \bottomrule
    \end{tabular}
\end{table}

\begin{table}[htbp]
    \centering
    \caption{Best to Others and Others to Worst fuzzy vectors for different examples.}
    \label{tab:fuzzy_vectors}
    \renewcommand{\arraystretch}{1.3}
    \begin{tabular}{ccc}
        \toprule
        \textbf{Examples} & \textbf{Best to Others vector $\tilde{A}_b$} & \textbf{Others to Worst vector $\tilde{A}_w$} \\
        \midrule
        Ex 3 & $(\tilde{2}, \tilde{1}, \tilde{4}, \tilde{3}, \tilde{8})$ & $(\tilde{3}, \tilde{8}, \tilde{5}, \tilde{4}, \tilde{1})^T$ \\
        Ex 4 & $(\tilde{3}, \tilde{1}, \tilde{4}, \tilde{2}, \tilde{8})$ & $(\tilde{3}, \tilde{8}, \tilde{5}, \tilde{4}, \tilde{1})^T$ \\
        Ex 5 & $(\tilde{3}, \tilde{1}, \tilde{4}, \tilde{2}, \tilde{6})$ & $(\tilde{2}, \tilde{6}, \tilde{6}, \tilde{3}, \tilde{1})^T$ \\
        Ex 6 & $(\tilde{4}, \tilde{1}, \tilde{3}, \tilde{2}, \tilde{6})$ & $(\tilde{2}, \tilde{6}, \tilde{6}, \tilde{3}, \tilde{1})^T$ \\
        Ex 7 & $(\tilde{3}, \tilde{1}, \tilde{3}, \tilde{2}, \tilde{6})$ & $(\tilde{2}, \tilde{6}, \tilde{5}, \tilde{3}, \tilde{1})^T$ \\
        Ex 8 & $(\tilde{3}, \tilde{1}, \tilde{4}, \tilde{2}, \tilde{6})$ & $(\tilde{2}, \tilde{6}, \tilde{5}, \tilde{3}, \tilde{1})^T$ \\
        Ex 9 & $(\tilde{4}, \tilde{1}, \tilde{3}, \tilde{2}, \tilde{6})$ & $(\tilde{2}, \tilde{6}, \tilde{5}, \tilde{3}, \tilde{1})^T$ \\
        Ex 10 & $(\tilde{4}, \tilde{1}, \tilde{3}, \tilde{2}, \tilde{6})$ & $(\tilde{3}, \tilde{6}, \tilde{4}, \tilde{5}, \tilde{1})^T$ \\
        \bottomrule
    \end{tabular}
\end{table}

\begin{figure}[htbp]
    \centering
    \includegraphics[width=0.8\textwidth]{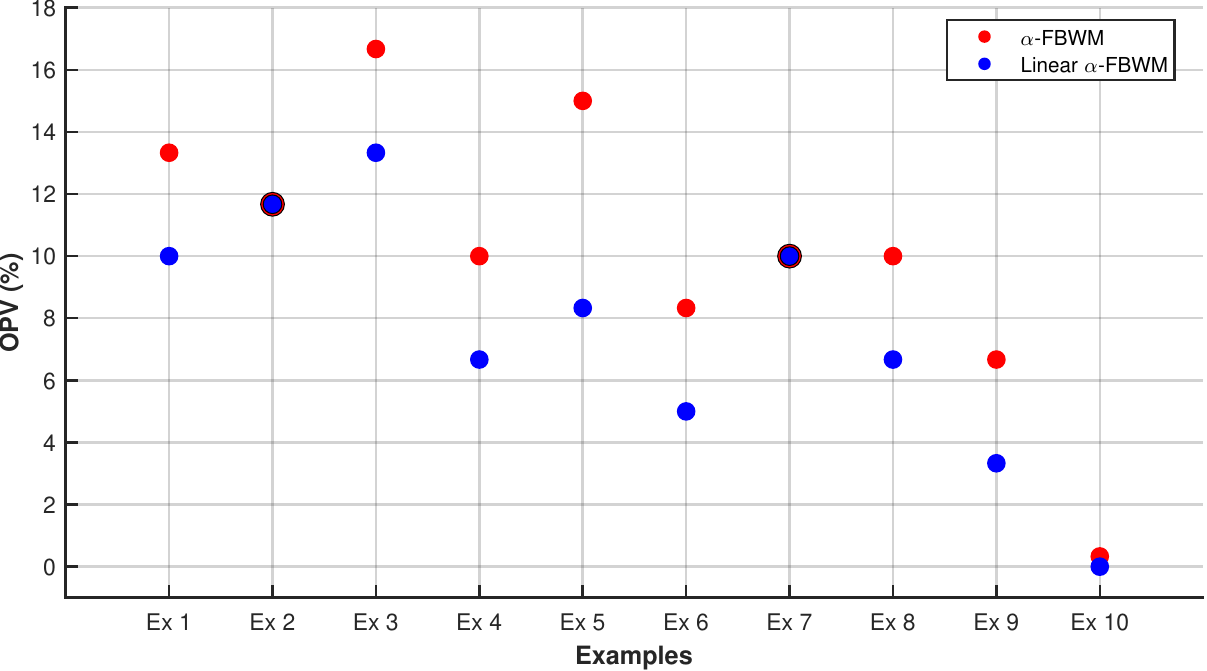}
    \caption{Comparison of OPV (\%) between \(\alpha\)-FBWM and Linear \(\alpha\)-FBWM across different examples.}
    \label{fig:opv_comparison}
\end{figure}

\subsection{Implementation of the Linear \texorpdfstring{$\alpha$}{alpha}-PFBWM: Country Size Estimation}
In this section, the practical implementation of the proposed Linear $\alpha$-PFBWM framework is demonstrated. To illustrate the evaluation process clearly, a decision case from the existing literature \cite{corrente2024better} is adopted to determine the surface areas of ten European countries. To validate the effectiveness of the Linear $\alpha$-PFBWM, the original crisp inputs are adapted by converting the initial country ratings provided by DM into Triangular Fuzzy Numbers (TFNs). Utilizing the systematic steps outlined below, the final priorities for all ten countries are computed.

\paragraph{\textbf{Step 1: Initial Fuzzy Ratings}}
Initially, a fuzzy rating is assigned to each of the ten countries under consideration: Bosnia, Estonia, Ireland, Iceland, Lithuania, the Czech Republic, Romania, Slovenia, Spain, and Switzerland. To transform the original crisp problem into a fuzzy environment, the reference ratings from the literature are utilized as the middle value ($m$) of the corresponding TFNs. Around this middle value, the lower and upper bounds are expanded to capture the subjective uncertainty inherent in estimating physical landmasses. Following this approach, Slovenia receives a reference fuzzy rating of $\tilde{r}_{\text{Slovenia}} = (1.0, 1.0, 1.0)$. The complete set of initial fuzzy ratings is compiled in Table~\ref{tab:step1_fuzzy_ratings}. 

\paragraph{\textbf{Step 2: Selection of Reference Countries}}
Subsequently, a subset of $t = 5$ reference countries, $\{ a_{\gamma_1}, a_{\gamma_2}, a_{\gamma_3}, a_{\gamma_4}\\, a_{\gamma_5} \}$, is selected from the complete set of ten countries. To guarantee that these reference points are well-distributed across the entire spectrum of country sizes, the MILP model presented in Model (\ref{PFBWM_MILP}) is solved. This formulation utilizes the middle values ($m_j$) derived in Step 1. Solving this optimization model identifies the five optimal representative reference countries: Spain $(a_{\gamma_5})$, Romania $(a_{\gamma_4})$, the Czech Republic $(a_{\gamma_3})$, Switzerland $(a_{\gamma_2})$, and Slovenia $(a_{\gamma_1})$.

\begin{table}[htbp]
    \centering
   \caption{Initial fuzzy ratings and final priority weights for all ten countries.}
    \label{tab:step1_fuzzy_ratings}
    \begin{tabular}{lccc}
        \toprule
        \textbf{Alternatives}& \textbf{Country} & \textbf{Fuzzy Rating} $\tilde{r}(\cdot)$ & $R(u(\tilde{r}(\cdot)))$\\
        \midrule
        $a_1$ & Bosnia & $(2.5, 3.0, 3.5)$ & $0.0912$ \\
        $a_2$ & Estonia & $(2.2, 2.7, 3.2)$ & $0.0918$ \\
        $a_3$ & Ireland & $(3.0, 3.5, 4.0)$ & $0.1023$\\
        $a_4$ & Iceland & $(4.2, 5.0, 5.8)$ & $0.3198$ \\
        $a_5$ & Lithuania & $(2.9, 3.7, 4.5)$ & $0.1291$\\
        $a_6$ & Czech Rep. & $(3.2, 4.0, 4.8)$ & $0.1268$ \\
        $a_7$ & Romania & $(5.2, 6.0, 6.8)$ & $0.2134$\\
        $a_8$ & Slovenia & $(1.0, 1.0, 1.0)$ & $0.0454$ \\
        $a_9$ & Spain & $(6.2, 7.0, 7.8)$ & $0.5231$ \\
        $a_{10}$ & Switzerland & $(2.2, 3.0, 3.8)$ & $0.0912$ \\
        \bottomrule
    \end{tabular}
\end{table}

\paragraph{\textbf{Step 3: Determining Priorities of Reference Countries}}
The objective of this step is to determine the priorities for the five selected reference countries. First, the best and worst reference options are identified by comparing the middle values ($m$) of their initial fuzzy ratings. Based on this comparison, Spain is designated as the best reference country and Slovenia as the worst reference country.

Next, the fuzzy pairwise comparison vectors are constructed. The Best-to-Other (BO) and Other-to-Worst (OW) preference vectors are obtained using the standard linguistic scale presented in Table~\ref{tab:linguistic_tfns}, yielding the fuzzy comparison data reported in Table~\ref{tab:fbo_fow_vectors}. Before deriving the priorities, the consistency of the pairwise judgments is verified by computing the $CR^{I}$ through Eq.~(\ref{CR}) and Eq.~(\ref{CR_j}). The computed consistency value stands at $0.2268$, which is well below the maximum acceptable threshold of $0.3062$ given in Table~\ref{tab:consistency_thresholds}. This confirms that the evaluations exhibit acceptable consistency. Finally, the proposed Linear $\alpha$-FBWM optimization model is solved to determine the optimal priorities of the reference countries, which are summarized in the lower section of Table~\ref{tab:fbo_fow_vectors}.

\begin{table*}[htbp]
    \centering
    \caption{Fuzzy pairwise comparisons and derived priorities for reference countries.}
    \label{tab:fbo_fow_vectors}
    \resizebox{\textwidth}{!}{
    \begin{tabular}{lccccc}
        \toprule
        & \textbf{Czech Rep.} ($a_{\gamma_3}$) & \textbf{Romania} ($a_{\gamma_4}$) & \textbf{Slovenia} ($a_{\gamma_1}$) & \textbf{Spain} ($a_{\gamma_5}$) & \textbf{Switzerland} ($a_{\gamma_2}$) \\
        \midrule
        $\tilde{a}_{BO}$ & $(4, 5, 6)$ & $(2, 3, 4)$ & $(9, 9, 9)$ & $(1, 1, 1)$ & $(6, 7, 8)$ \\[1ex]
        $\tilde{a}_{OW}$ & $(4,5,6)$ & $(6,7,8)$ & $(1, 1, 1)$ & $(9,9,9)$ & $(2,3,4)$ \\[1ex]
        $\tilde{r}(\cdot)$ & $\tilde{r}(a_{\gamma_3}) = (3.2, 4.0, 4.8)$ & $\tilde{r}(a_{\gamma_4}) = (5.2, 6.0, 6.8)$ & $\tilde{r}(a_{\gamma_1}) = (1.0, 1.0, 1.0)$ & $\tilde{r}(a_{\gamma_5}) = (6.2, 7.0, 7.8)$ & $\tilde{r}(a_{\gamma_2}) = (2.2, 3.0, 3.8)$ \\[1ex]
        $\tilde{u}(r(\cdot))$  & $(0.1062, 0.1240, 0.1586)$ & $(0.1586, 0.2040, 0.3058)$ & $(0.0454, 0.0454, 0.0454)$ & $(0.5231, 0.5231, 0.5231)$ & $(0.0790, 0.0908, 0.1051)$ \\
        \bottomrule
    \end{tabular}
    }
\end{table*}

\paragraph{\textbf{Step 4: Computation of Priorities for All Non-Reference Countries}}
In the final step, the priorities for non-reference countries are derived by employing the fuzzy interpolation formula defined in Eq.~(\ref{fuzzy_interpolation}) alongside the reference priorities determined in Step 3. To illustrate this procedure, consider the evaluation for Bosnia ($a_1$). Bosnia holds an initial fuzzy rating of $\tilde{r}(a_1) = (2.5, 3.0, 3.5)$. Its middle value is $m_1 = 3.0$, which strictly falls within the reference interval $[m_{\gamma_2}, m_{\gamma_3}] = [3.0, 4.0]$ defined by Switzerland and the Czech Republic. Consequently, the final priority $R(\tilde{u}(\tilde{r}(a_1)))$ for Bosnia is evaluated using the fuzzy interpolation formula as follows:

\vspace{0.2cm}

$\begin{aligned}
 R(\tilde{u}(\tilde{r}(a_1))) & = R(\tilde{u}(\tilde{r}(a_{\gamma_2}))) + R\left( \frac{\tilde{u}(\tilde{r}(a_{\gamma_3})) - \tilde{u}(\tilde{r}(a_{\gamma_2}))}{\tilde{r}(a_{\gamma_3}) - \tilde{r}(a_{\gamma_2})} \times (\tilde{r}(a_1) - \tilde{r}(a_{\gamma_2})) \right) \\[15pt]
&= R(0.0790, 0.0908, 0.1051) + R \left[ \frac{(0.1062, 0.1240, 0.1586) - (0.0790, 0.0908, 0.1051)}{(3.2, 4, 4.8) - (2.2, 3.0, 3.8)} \right. \\
& \quad \quad \quad \quad \quad \quad \quad \quad \quad \quad \quad \quad \quad \quad \times ((2.5, 3, 3.5) - (3.2, 4, 4.8)) \Biggr] \\[15pt]
&= 0.0912 + R(-0.17247, 0, 0.17247) \\[10pt]
&= 0.0912 + 0 \\[10pt]
&= 0.0912.
\end{aligned}$

\vspace{0.2cm}

Similarly, the final priorities for the remaining non-reference countries are determined by applying the procedure discussed above. In the final compilation presented in Table~\ref{tab:step1_fuzzy_ratings}, the priorities of the reference countries originally derived as TFNs, have been defuzzified into crisp values for a standardized presentation. Consequently, due to the structure of the fuzzy interpolation formulation, the priorities for all non-reference alternatives are directly obtained as crisp values. The complete set of derived final priorities for all ten countries is summarized in Table~\ref{tab:step1_fuzzy_ratings}.

\section{Industrial Application: Warehouse Location Selection}

To demonstrate the practical applicability, computational efficiency, and scalability of the proposed Linear $\alpha$-PFBWM, a real-world large-scale MCDM problem is investigated in this section.

\subsection{Problem Description and Expert Selection}

To optimize its supply chain and distribution network, a leading multinational paint manufacturing company in India aimed to establish new strategic warehouses across Gujarat. The expansion strategy required a combination of local depots near major urban hubs for high-volume markets and upcountry depots in interior regions to penetrate rural markets. Therefore, instead of selecting a single facility, the objective was to evaluate a large pool of potential sites to identify the top five locations for a multi-tiered distribution network. Evaluating these facilities represents a complex Multi-Criteria Decision-Making (MCDM) problem involving inherent uncertainties, conflicting logistical criteria, and a large number of alternatives. In corporate practice, while final strategic authorizations are made by top-level supply chain management (SCM) teams, the foundational evaluation depends heavily on localized operational data. A senior commercial executive accountable for regional operations and infrastructure assessment was consulted as the decision-maker (DM) to guarantee an accurate review. Although the executive does not grant final corporate authorization, their field-level insights into daily depot operations and vendor management provided crucial primary evaluation ratings. These inputs were then used in the proposed Linear $\alpha$-PFBWM framework. This approach connects ground-level operations with high-level decisions, helping senior management make a clear, data-driven choice.

\subsection{Selection of Alternatives and Criteria}
Following a preliminary geographic and industrial screening, 20 logistics hubs and industrial estates across Gujarat were selected as potential warehouse locations. The set of alternatives is defined as $A = \{A_1, A_2, \dots, A_{20}\}$, where the locations correspond to: $A_1$: Aslali, $A_2$: Changodar, $A_3$: Sanand, $A_4$: Bavla, $A_5$: Naroda, $A_6$: Vapi, $A_7$: Ankleshwar, $A_8$: Hazira, $A_9$: Palsana, $A_{10}$: Dahej, $A_{11}$: Halol, $A_{12}$: Savli, $A_{13}$: Makarpura, $A_{14}$: Shapar Veraval, $A_{15}$: Metoda, $A_{16}$: Morbi, $A_{17}$: Sikka, $A_{18}$: Mundra, $A_{19}$: Kandla, and $A_{20}$: Gandhidham.

The DM evaluated these 20 sites against five primary criteria:
\begin{itemize}
    \item \textbf{$C_1$ (Transportation Connectivity):} Proximity to national highways, railway freight corridors, and ports, directly influencing inbound and outbound freight efficiency.
    \item \textbf{$C_2$ (Real Estate and Setup Cost):} Expenses related to land acquisition, leasing, structural construction, and initial operational deployment.
    \item \textbf{$C_3$ (Proximity to Market and Dealers):} Distance to major consumer clusters and established distributor networks to minimize delivery lead times.
    \item \textbf{$C_4$ (Infrastructure Quality):} Availability and reliability of core utilities, including uninterrupted power, water supply, and internal road networks within the industrial estates.
    \item \textbf{$C_5$ (Labor Availability and Safety):} Access to skilled and semi-skilled labor pools, alongside safety and environmental compliance necessary for chemical and paint storage.
\end{itemize}

\subsection{Data Collection via Three-Point Estimation Method}
To capture the uncertainty in expert judgments, the DM evaluated the alternatives using the three-point estimation method \cite{3point}. This technique reduces the cognitive load on the expert by allowing evaluations to be expressed through three distinct values on a continuous scale from 1 to 100. Specifically, for each alternative under a given criterion, the DM provides:
\begin{enumerate}
    \item a pessimistic or minimum score $l$, representing the worst-case scenario,
    \item a most likely or moderate score $m$, representing the expected condition, and
    \item an optimistic or maximum score $u$, representing the best-case scenario.
\end{enumerate}
These three estimates are mapped directly to establish the lower, middle, and upper values of a Triangular Fuzzy Number (TFN), denoted as $\tilde{A} = (l, m, u)$, where $l \le m \le u$. This approach provides a continuous representation of judgment uncertainty while keeping the initial data collection phase simple.

\subsection{Weights of Criteria Using \texorpdfstring{$\alpha$}{alpha}-FBWM and Priorities of Alternatives Using \texorpdfstring{$\alpha$}{alpha}-PFBWM}

The application of the proposed framework, which integrates the Linear $\alpha$-FBWM for determining the weights of criteria and the Linear $\alpha$-PFBWM for deriving the priorities of alternatives, is structured into the following sequential phases:

\paragraph{Phase 1: Determination of the Weights of Criteria}
First, the DM identified the best (most important) and worst (least important) criteria from the set $\{C_1, C_2, C_3, C_4, C_5\}$. Based on this assessment, \textbf{Labor Availability and Safety ($C_5$)} was selected as the best criterion and \textbf{Real Estate and Setup Cost ($C_2$)} as the worst. 

The DM then provided fuzzy pairwise comparisons for the Best-to-Others (BO) and Others-to-Worst (OW) vectors using the linguistic scale defined in Table~\ref{tab:linguistic_tfns}. To validate these evaluations prior to weight calculation, the Input-based Consistency Ratio ($CR^I$) was computed. The resulting $CR^I$ value fell below the acceptable threshold, verifying the consistency of the DM's pairwise judgments. 

Finally, the proposed Linear $\alpha$-FBWM was applied to these vectors to determine the weights of the five criteria. The complete fuzzy pairwise comparison vectors (BO and OW), the resulting optimal fuzzy weights, and their defuzzified values are summarized in Table~\ref{tab:criteria_comparison}.

\begin{table}[htbp]
\centering
\caption{Fuzzy pairwise comparisons and the derived optimal weights for the evaluation criteria.}
\label{tab:criteria_comparison}
\renewcommand{\arraystretch}{1.2}
\resizebox{\textwidth}{!}{%
\begin{tabular}{lccccc}
\hline
\textbf{Criteria} & $C_1$ & $C_2$ & $C_3$ & $C_4$ & $C_5$ \\ \hline
\textbf{Best-to-Others (BO)}& (4,5,6) & (5,6,7) & (2,3,4) & (1,2,3) & (1,1,1) \\ \hline
\textbf{Others-to-Worst (OW)} & (1,2,3) & (1,1,1) & (3,4,5) & (4,5,6) & (5,6,7) \\ \hline

\textbf{Fuzzy Weights} & \textbf{$(0.0889,0.0988,0.1118)$} & \textbf{$(0.0564,0.0588,0.0622)$} & \textbf{$(0.1298,0.1642,0.2124)$} & \textbf{$(0.1778,0.2328,0.3981)$} & \textbf{$(0.3823,0.4240,0.4661)$} \\ \hline
\textbf{Defuzzified Weights $({W_i})$} & \textbf{$0.0993$} & \textbf{$0.0590$} & \textbf{$0.1665$} & \textbf{$0.2512$} & \textbf{$0.4240$} \\ \hline
\end{tabular}%
}
\end{table}

\paragraph{Phase 2: Fuzzy Rating of Alternatives}
In this phase, the DM evaluated the performance of the 20 warehouse locations against the five criteria. Using the three-point estimation method, the DM assigned a pessimistic ($l$), most likely ($m$), and optimistic ($u$) score on a scale from 1 to 100 for each alternative under every criterion. These estimates were modeled as Triangular Fuzzy Numbers (TFNs) to establish the initial performance ratings, which are detailed in Table~\ref{tab:fuzzy_ratings}.

\begin{table}[htbp]
\centering
\caption{Initial fuzzy ratings of 20 warehouse locations across five criteria (Scale: 1-100).}
\label{tab:fuzzy_ratings}
\resizebox{\textwidth}{!}{
\renewcommand{\arraystretch}{1.1}
\begin{tabular}{lccccc}
\hline
\textbf{Alternatives} & \textbf{$C_1$ (Transport)} & \textbf{$C_2$ (Cost)} & \textbf{$C_3$ (Market)} & \textbf{$C_4$ (Infra)} & \textbf{$C_5$ (Labor)} \\ \hline
$A_1$: Aslali (Ahmedabad) & (78, 83, 89) & (48, 55, 63) & (91, 96, 99) & (74, 80, 86) & (72, 78, 85) \\ 
$A_2$: Changodar (Ahmedabad)& (72, 79, 85) & (52, 60, 68) & (88, 93, 97) & (78, 84, 90) & (75, 81, 88) \\ 
$A_3$: Sanand (Ahmedabad)   & (82, 88, 93) & (35, 42, 50) & (80, 85, 90) & (92, 97, 100)& (68, 74, 81) \\ 
$A_4$: Bavla (Ahmedabad)    & (65, 71, 78) & (68, 75, 83) & (74, 80, 86) & (62, 68, 75) & (60, 67, 74) \\ 
$A_5$: Naroda (Ahmedabad)   & (75, 81, 87) & (45, 52, 60) & (93, 98, 100)& (70, 76, 82) & (70, 76, 83) \\ 
$A_6$: Vapi GIDC      & (85, 91, 96) & (28, 35, 43) & (72, 78, 84) & (90, 95, 99) & (82, 88, 94) \\ 
$A_7$: Ankleshwar GIDC& (80, 86, 91) & (38, 45, 52) & (70, 76, 82) & (85, 90, 95) & (78, 84, 91) \\ 
$A_8$: Hazira (Surat)& (90, 96, 100)& (20, 28, 36) & (65, 71, 77) & (88, 93, 98) & (64, 71, 78) \\ 
$A_9$: Palsana (Surat)& (74, 80, 86) & (58, 65, 73) & (82, 87, 92) & (76, 82, 88) & (80, 86, 92) \\ 
$A_{10}$: Dahej (Bharuch)& (88, 94, 99) & (32, 40, 48) & (60, 67, 74) & (91, 96, 100)& (66, 73, 80) \\ 
$A_{11}$: Halol GIDC  & (68, 74, 80) & (65, 72, 80) & (68, 74, 80) & (72, 78, 85) & (72, 78, 85) \\ 
$A_{12}$: Savli (Vadodara) & (70, 76, 83) & (60, 68, 76) & (75, 81, 87) & (78, 84, 90) & (75, 81, 88) \\ 
$A_{13}$: Makarpura(Vadodara)& (76, 82, 88) & (50, 58, 66) & (80, 86, 91) & (74, 80, 87) & (78, 84, 90) \\ 
$A_{14}$: Shapar-Veraval& (62, 68, 75) & (70, 78, 85) & (65, 71, 78) & (60, 66, 73) & (65, 72, 79) \\ 
$A_{15}$: Metoda GIDC & (65, 71, 78) & (68, 76, 84) & (68, 74, 80) & (65, 71, 78) & (68, 75, 82) \\ 
$A_{16}$: Morbi GIDC  & (72, 78, 85) & (55, 62, 70) & (58, 65, 72) & (78, 84, 91) & (80, 87, 93) \\ 
$A_{17}$: Sikka (Jamnagar)& (82, 88, 93) & (48, 56, 64) & (55, 61, 68) & (82, 88, 93) & (60, 66, 73) \\ 
$A_{18}$: Mundra (Kutch)& (93, 98, 100)& (25, 32, 40) & (50, 56, 62) & (91, 96, 100)& (62, 69, 76) \\ 
$A_{19}$: Kandla (Kutch)& (91, 97, 100)& (30, 38, 46) & (52, 58, 64) & (88, 93, 98) & (64, 71, 78) \\ 
$A_{20}$: Gandhidham  & (80, 86, 92) & (52, 60, 68) & (58, 64, 70) & (80, 86, 92) & (66, 73, 80) \\ \hline
\end{tabular}
}
\end{table}

\paragraph{Phase 3: Selection of Reference Alternatives}
In this phase, a subset of $t=5$ reference alternatives was determined for each criterion from the total pool of 20 alternatives. This selection follows the optimization procedure outlined in Step 2 of Section 3.4, which requires solving the MILP model (\ref{fuzzy_interpolation}). Applying this model yielded the following optimal sets of reference alternatives for each criterion:
\begin{itemize}
    \item $C_1$: $\{A_1, A_6, A_{12}, A_{14}, A_{18}\}$
    \item $C_2$: $\{A_5, A_8, A_9, A_{10}, A_{14}\}$
    \item $C_3$: $\{A_3, A_5, A_{11}, A_{16}, A_{18}\}$
    \item $C_4$: $\{A_3, A_5, A_7, A_{14}, A_{16}\}$
    \item $C_5$: $\{A_2, A_6, A_5, A_{17}, A_{19}\}$
\end{itemize}

\paragraph{Phase 4: Priorities of Reference Alternatives}
In this phase, the DM performed fuzzy pairwise comparisons among the selected reference alternatives for each criterion. The Best-to-Others (BO) and Others-to-Worst (OW) vectors were constructed using the linguistic scale defined in Table~\ref{tab:linguistic_tfns}. 

To ensure the reliability of the responses, the Input-based Consistency Ratio ($CR^I$) was calculated for each pairwise comparison vectors. In instances where the $CR^I$ exceeded the acceptable threshold, the DM's judgments were iteratively revised until satisfactory consistency was achieved. The final consistent fuzzy pairwise comparison vectors are presented in Table~\ref{tab:master_bwm_table}. Subsequently, the proposed Linear $\alpha$-FBWM was applied to these vectors to compute the priorities of the reference alternatives, which are also presented in Table~\ref{tab:master_bwm_table}.

\begin{table}[htbp]
    \centering
    \caption{Comprehensive Fuzzy Pairwise Comparisons and Priorities for Reference Alternatives}
    \label{tab:master_bwm_table}
    \resizebox{\textwidth}{!}{
    \begin{tabular}{clccccc}
        \toprule
        
        \multirow{4}{*}{$C_1$} 
        & Reference Alternatives & $A_{14}$ & $A_{12}$ & $A_{1}$ & $A_{6}$ & $A_{18}$ \\
        & $A_b$ & $(7,8,9)$ & $(6,7,8)$ & $(4,5,6)$ & $(2,3,4)$ & $(1,1,1)$ \\
        & $A_w$ & $(1,1,1)$ & $(1,2,3)$ & $(2,3,4)$ & $(5,6,7)$ & $(7,8,9)$ \\
       & Priorities & $(0.0521, 0.0521, 0.0529)$ & $(0.0845, 0.0893, 0.0960)$ & $(0.1121, 0.1248, 0.1426)$ & $(0.1605, 0.2083, 0.2669)$ & $(0.4745, 0.5214, 0.5736)$ \\
        \midrule
        
        \multirow{4}{*}{$C_2$} 
        & Reference Alternatives & $A_{8}$ & $A_{10}$ & $A_{5}$ & $A_{9}$ & $A_{14}$ \\
        & $A_b$ & $(7,8,9)$ & $(4,5,6)$ & $(2,3,4)$ & $(1,2,3)$ & $(1,1,1)$ \\
        & $A_w$ & $(1,1,1)$ & $(3,4,5)$ & $(4,5,6)$ & $(6,7,8)$ & $(7,8,9)$ \\
        & Priorities & $(0.0404, 0.0435, 0.0435)$ & $(0.0887, 0.1009, 0.1180)$ & $(0.1287, 0.1703, 0.2237)$ & $(0.1767, 0.2374, 0.4075)$ & $(0.3885, 0.4319, 0.4480)$ \\
        \midrule
        
        \multirow{4}{*}{$C_3$} 
        & Reference Alternatives & $A_{18}$ & $A_{16}$ & $A_{11}$ & $A_{3}$ & $A_{5}$\\
        & $A_b$ & $(7,8,9)$ & $(4,5,6)$ & $(2,3,4)$ & $(1,2,3)$ & $(1,1,1)$ \\
        & $A_w$ & $(1,1,1)$ & $(2,3,4)$ & $(4,5,6)$ & $(6,7,8)$ & $(7,8,9)$ \\
        & Priorities & $(0.0437, 0.0437, 0.0437)$ & $(0.0912, 0.1013, 0.1144)$ & $(0.1346, 0.1676, 0.2193)$ & $(0.1837, 0.2328, 0.4148)$ & $(0.3850, 0.4288, 0.4725)$ \\
        \midrule
        
        \multirow{4}{*}{$C_4$} 
        & Reference Alternatives & $A_{14}$ & $A_{5}$ & $A_{16}$ & $A_{7}$ & $A_{3}$\\
        & $A_b$ & $(6,7,8)$ & $(3,4,5)$ & $(2,3,4)$ & $(1,2,3)$ & $(1,1,1)$ \\
        & $A_w$ & $(1,1,1)$ & $(2,3,4)$ & $(3,4,5)$ & $(4,5,6)$ & $(6,7,8)$ \\
        & Priorities & $(0.0500, 0.0539, 0.0553)$ & $(0.1015, 0.1171, 0.1460)$ & $(0.1167, 0.1617, 0.2018)$ & $(0.1673, 0.2281, 0.3538)$ & $(0.3858, 0.4312, 0.4538)$ \\
        \midrule
        
        \multirow{4}{*}{$C_5$} 
        & Reference Alternatives & $A_{17}$ & $A_{19}$ & $A_{5}$ & $A_{2}$ & $A_{6}$ \\
        & $A_b$ & $(6,7,8)$ & $(4,5,6)$ & $(2,3,4)$ & $(1,2,3)$ & $(1,1,1)$ \\
        & $A_w$ & $(1,1,1)$ & $(2,3,4)$ & $(3,4,5)$ & $(4,5,6)$ & $(6,7,8)$ \\
       & Priorities & $(0.0465, 0.0531, 0.0576)$ & $(0.0842, 0.0982, 0.1195)$ & $(0.1216, 0.1665, 0.2253)$ & $(0.1638, 0.2399, 0.3456)$ & $(0.4121, 0.4386, 0.4386)$ \\
        \bottomrule
    \end{tabular}
    }
\end{table}

\paragraph{Phase 5: Priority Computation for Non-Reference Alternatives and Final Ranking}
In this final phase, the priorities of non-reference alternatives were determined for each criterion based on the reference alternatives' priorities, following the fuzzy interpolation procedure detailed in Step 4 of Section 3.4. The resulting priorities are presented in Table~\ref{tab:master_ranking_matrix}. To establish the final overall performance, a total utility score was calculated for each alternative by multiplying the criteria weights by their corresponding priorities and aggregating the results. 

Based on the final scores, the overall ranking of the 20 warehouse locations was determined, identifying Vapi ($A_6$), Palsana ($A_9$), Morbi ($A_{16}$), Ankleshwar ($A_{7}$), and Makarpura ($A_{13}$) as the top five strategic alternatives. To highlight the efficiency of this selection, a comparative evaluation of the decision-making workload is conducted. In a standard Fuzzy Best-Worst Method (FBWM), evaluating $m=20$ alternatives across $n=5$ criteria requires $2m - 3 = 37$ fuzzy pairwise comparisons per criterion, totaling $185$ comparisons. In contrast, the proposed parsimonious framework ($\alpha$-PFBWM) reduces this cognitive burden by utilizing a subset of $t=5$ reference alternatives. This reduces the fuzzy comparisons per criterion to $2t - 3 = 7$, requiring only 35 comparisons alongside the initial 20 fuzzy ratings of alternatives. While the criteria-level comparisons remain identical in both approaches, this parsimonious expansion achieves an $81.08\%$ reduction in the fuzzy pairwise comparison workload for alternatives.

From a corporate perspective, this substantial workload reduction offers distinct operational advantages. For the regional executive (DM), executing 185 fuzzy pairwise comparisons could introduce response bias due to cognitive fatigue. By requiring only 20 initial fuzzy ratings and 35 fuzzy pairwise comparisons, the Linear $\alpha$-PFBWM minimizes the evaluation burden while capturing high-quality primary data. Consequently, this consistent dataset provides senior management with reliable, data-driven inputs rather than inconsistent or biased opinions from different branches. By connecting empirical, ground-level data with long-term corporate strategy, the framework enhances decision-making confidence. This provides a clear operational justification for authorizing the multi-tiered warehouse network, balancing local depots for high-volume urban demand with upcountry depots to penetrate deeper rural markets.

\begin{table}[htbp]
\centering
\caption{Final decision matrix: Interpolated priorities of non-reference alternatives, final crisp scores, and final ranking of warehouse locations.}
\label{tab:master_ranking_matrix}
\renewcommand{\arraystretch}{1.3}
\resizebox{\textwidth}{!}{
\begin{tabular}{lcccccccc}
\hline
\textbf{Alt. $\setminus$ Criteria} & \textbf{$C_1$} & \textbf{$C_2$} & \textbf{$C_3$} & \textbf{$C_4$} & \textbf{$C_5$} & \textbf{Final Crisp Score} & \textbf{Rank} \\ 
\textbf{Weights $\rightarrow$} &  \textbf{$0.0993$} & \textbf{$0.0590$} & \textbf{$0.1665$} & \textbf{$0.2512$} & \textbf{$0.4240$}  & $U_i = \sum (w_j \cdot v_{ij})$ & \\ \hline
$A_1$: Aslali       & 0.1256 & 0.1915 & 0.4020 & 0.1401 & 0.1992 & 0.2104 & 9 \\
$A_2$: Changodar    & 0.1050 & 0.2236 & 0.3619 & 0.1609 & 0.2448 & 0.2281 & 7 \\
$A_3$: Sanand       & 0.1784 & 0.1135 & 0.2549 & 0.4274 & 0.1384 & 0.2329 & 6 \\
$A_4$: Bavla        & 0.0662 & 0.3877 & 0.2166 & 0.0666 & 0.0621 & 0.1086 & 20 \\
$A_5$: Naroda       & 0.1153 & 0.1723 & 0.4288 & 0.1193 & 0.1688 & 0.1946 & 12 \\
$A_6$: Vapi         & 0.2101 & 0.0772 & 0.2013 & 0.3739 & 0.4340 & 0.3369 & 1 \\
$A_7$: Ankleshwar   & 0.1573 & 0.1311 & 0.1860 & 0.2389 & 0.3259 & 0.2525 & 4 \\
$A_8$: Hazira       & 0.4331 & 0.0430 & 0.1477 & 0.3189 & 0.0994 & 0.1924 & 13 \\
$A_9$: Palsana      & 0.1102 & 0.2556 & 0.2817 & 0.1505 & 0.3800 & 0.2719 & 2 \\
$A_{10}$: Dahej     & 0.3439 & 0.1017 & 0.1171 & 0.3988 & 0.1235 & 0.2122 & 8 \\
$A_{11}$: Halol     & 0.0802 & 0.3481 & 0.1707 & 0.1297 & 0.1992 & 0.1740 & 15 \\
$A_{12}$: Savli     & 0.0896 & 0.2952 & 0.2243 & 0.1609 & 0.2448 & 0.2079 & 10 \\
$A_{13}$: Makarpura & 0.1205 & 0.2107 & 0.2683 & 0.1401 & 0.3259 & 0.2424 & 5 \\
$A_{14}$: Shapar    & 0.0522 & 0.4273 & 0.1477 & 0.0535 & 0.1087 & 0.1145 & 19 \\
$A_{15}$: Metoda    & 0.0662 & 0.4009 & 0.1707 & 0.0864 & 0.1536 & 0.1455 & 16 \\
$A_{16}$: Morbi     & 0.0999 & 0.2364 & 0.1018 & 0.1609 & 0.4070 & 0.2538 & 3 \\
$A_{17}$: Sikka     & 0.1784 & 0.1979 & 0.0760 & 0.2129 & 0.0528 & 0.1179 & 18 \\
$A_{18}$: Mundra    & 0.5223 & 0.0626 & 0.0437 & 0.3988 & 0.0808 & 0.1973 & 11 \\
$A_{19}$: Kandla    & 0.4777 & 0.0919 & 0.0566 & 0.3189 & 0.0994 & 0.1845 & 14 \\
$A_{20}$: Gandhidham& 0.1573 & 0.2236 & 0.0953 & 0.1869 & 0.1235 & 0.1440 & 17 \\ \hline
\end{tabular}
}
\end{table}

\section{Conclusion and Future Research Directions}

In Multi-Criteria Decision-Making (MCDM) under uncertain environments, the Fuzzy Best-Worst Method (FBWM) is a widely adopted approach for determining criteria weights. This study proposed the Linear $\alpha$-FBWM to overcome the computational and structural limitations of the existing $\alpha$-FBWM. This linear model was subsequently extended into a parsimonious framework, termed the Linear $\alpha$-PFBWM, designed to handle large-scale applications with numerous alternatives. The proposed approach effectively resolves the methodological shortcomings identified in existing parsimonious FBWM extensions. The primary contributions and findings of this research are summarized below:

\begin{itemize}
    \item \textbf{Linear $\alpha$-FBWM Model Formulation:} We have proposed a novel Linear $\alpha$-FBWM model, which is developed by converting the existing $\alpha$-FBWM model into a linear form. Since finding the optimal weight set from the Linear $\alpha$-FBWM was difficult due to the inclusion of infinitely many constraints in the optimization model, we approximated the optimal weights using finite partitions of the interval $[0,1]$. Furthermore, the Error of Approximation (EA) for these approximate weights has been estimated.
    \item \textbf{Preserving Uncertainty and Ensuring Uniqueness:} We have shown that for a given finite partition $F$ of the interval $[0,1]$, the proposed Linear $\alpha$-FBWM model always gives a unique optimal solution. In this unique optimal solution, the weights of the criteria are obtained directly as Triangular Fuzzy Numbers (TFNs). By doing this, we overcome several limitations of the existing $\alpha$-FBWM model, where one had to solve $2n+1$ non-linear optimization models, followed by subsequent interval boundary averaging to determine a unique weight set for $n$ criteria, which was computationally intensive and time-consuming. Moreover, the unique weights obtained from the existing model were crisp values, meaning the valuable fuzzy uncertainty hidden within the criteria weights was completely lost, whereas Linear $\alpha$-FBWM resolves these limitations.
    
    \item \textbf{A Novel Cognitive Metric (OPV):} We introduced a new metric called Ordinal Preference Violation (OPV), which is based on a cognitive psychology-driven concept known as the prominence effect. This metric assesses the alignment between the final weight rankings and DM's intal fuzzy pairwise preferences. It does this by conducting a realistic asymmetric evaluation to measure the discrepancy whenever the weights calculated by the model contradict the initial rankings set by the DM through the Best-to-Others (BO) and Others-to-Worst (OW) fuzzy pairwise comparison vectors. We have shown through different examples that the defuzzified weights obtained from the Linear $\alpha$-FBWM match the DM's initial preference rankings much better compared to the unique weights obtained from the existing $\alpha$-FBWM. 

    \item \textbf{A New Parsimonious Fuzzy Best-Worst Method (Linear $\alpha$-PFBWM):} We proposed the Linear $\alpha$-PFBWM specifically to handle complex decision-making problems in an uncertain environment, especially when the number of alternatives becomes very large. Designed to overcome the limitations of existing parsimonious FBWM frameworks, this method works by directly integrating the proposed Linear $\alpha$-FBWM model into its structure. This framework allows DMs to input alternative ratings directly as fuzzy numbers (TFNs) and calculates the priorities of non-reference alternatives using fuzzy interpolation without any early defuzzification of the priorities of reference alternatives. 
    
The scaling capabilities and practical usefulness of the proposed Linear $\alpha$-PFBWM were tested using a literature example and a real-world industrial case study. In practice, the framework required only 20 initial fuzzy ratings and 35 pairwise comparisons instead of the traditional 185. This resulted in an $81.08\%$ reduction in the main evaluation workload. This big decrease in the required comparisons shows that the model can gather high-quality DM insights easily without causing mental fatigue in large-scale decision problems.

\end{itemize}

\subsection{Limitations}
Despite its computational efficiency, the proposed Linear $\alpha$-PFBWM framework exhibits a structural limitation regarding the output format of the computed priorities. While the proposed Linear $\alpha$-FBWM model calculates the priorities of reference alternatives as Triangular Fuzzy Numbers (TFNs), the fuzzy interpolation formula yields crisp values for the non-reference alternatives. Consequently, the framework does not fully preserve fuzzy uncertainty across all final priorities, restricting their direct integration into subsequent fuzzy decision-making stages.

\subsection{Future Research Directions}
While proposed Linear $\alpha$-FBWM and Linear $\alpha$-PFBWM models achieve linear optimization and lower cognitive load, they open up several directions for future work:
\begin{itemize}
    \item A useful direction for future work would be to explore alternative solution methods or algorithms for solving model (\ref{linearformulationoptimization}) to see if there are other efficient ways to determine the optimal weights.
    \item The current Linear $\alpha$-FBWM framework is entirely designed around Triangular Fuzzy Numbers (TFNs). Future studies can extend this linear $\alpha$-cut approach to more complex uncertainty environments, such as Intuitionistic, Spherical, or Neutrosophic fuzzy sets.
    \item In the proposed Linear $\alpha$-PFBWM, the initial fuzzy ratings are restricted to TFNs. Future researchers can investigate how using other types of fuzzy numbers for initial ratings affects this parsimonious concept and how the fuzzy interpolation formula can be adapted accordingly.
    \item The proposed Linear $\alpha$-PFBWM currently computes alternatives priorities based on the judgments of a single DM. Developing an effective framework to aggregate judgments and work with multiple DMs in a group environment would be a highly valuable addition to this research.
\end{itemize}

\section*{Acknowledgements}
The first author gratefully acknowledges the Council of Scientific and Industrial Research (CSIR), New Delhi, India, for providing financial support through the Junior Research Fellowship (JRF-NET) (File No. 09/1274(21444)/2025-EMR-I) to carry out this research work.

\bibliographystyle{elsarticle-num}
\bibliography{ref}

\begin{thebibliography}{10}
\expandafter\ifx\csname url\endcsname\relax
  \def\url#1{\texttt{#1}}\fi
\expandafter\ifx\csname urlprefix\endcsname\relax\def\urlprefix{URL }\fi
\expandafter\ifx\csname href\endcsname\relax
  \def\href#1#2{#2} \def\path#1{#1}\fi

\bibitem{saaty1990make}
T.~L. Saaty, How to make a decision: the analytic hierarchy process, European Journal of Operational Research 48~(1) (1990) 9--26.

\bibitem{ANP}
T.~L. Saaty, Decision making with dependence and feedback: The analytic network process, RWS Publications, Pittsburgh, PA, 2001.

\bibitem{TOPSIS}
C.-L. Hwang, K.~Yoon, Methods for multiple attribute decision making, in: Multiple Attribute Decision Making: Methods and Applications, Springer, Berlin, Heidelberg, 1981, pp. 58--191.

\bibitem{VIKOR}
S.~Opricovic, G.-H. Tzeng, Compromise solution by {MCDM} methods: A comparative analysis of {VIKOR} and {TOPSIS}, European Journal of Operational Research 156~(2) (2004) 445--455.

\bibitem{ELECTRE}
B.~Roy, The outranking approach and the foundations of {ELECTRE} methods, in: Readings in Multiple Criteria Decision Aid, Springer, Berlin, Heidelberg, 1990, pp. 155--183.

\bibitem{PROMETHEE}
J.-P. Brans, P.~Vincke, A preference ranking organisation method: (the {PROMETHEE} method for multiple criteria decision-making), Management Science 31~(6) (1985) 647--656.

\bibitem{rezaei2015best}
J.~Rezaei, Best-worst multi-criteria decision-making method, Omega 53 (2015) 49--57.

\bibitem{corrente2024better}
S.~Corrente, S.~Greco, J.~Rezaei, Better decisions with less cognitive load: The {P}arsimonious {BWM}, Omega 126 (2024) 103075.

\bibitem{moslem2023novel}
S.~Moslem, A novel parsimonious best worst method for evaluating travel mode choice, IEEE Access 11 (2023) 16768--16773.

\bibitem{guo2017fuzzy}
S.~Guo, H.~Zhao, Fuzzy best-worst multi-criteria decision-making method and its applications, Knowledge-Based Systems 121 (2017) 23--31.

\bibitem{ratandhara2024alpha}
H.~M. Ratandhara, M.~Kumar, An $\alpha$-cut intervals based {Fuzzy Best-Worst Method} for {Multi-Criteria Decision-Making}, Applied Soft Computing 159 (2024) 111625.

\bibitem{fuzzyparsimoniousZBWM}
S.~Moslem, Evaluating commuters' travel mode choice using the {Z}-number extension of {P}arsimonious {B}est {W}orst {M}ethod, Applied Soft Computing 173 (2025) 112918.

\bibitem{fuzzyPBWMSDGs}
F.~Ecer, {\.I}.~Yaran~{\"O}gel, V.~G{\"o}{\c{c}}o{\u{g}}lu, Interlinking {SDGs} and cittaslow criteria for sustainable local decision prioritization: A parsimonious fuzzy best-worst method, Sustainable Development (2026).

\bibitem{definationsoffuzzyzimmermann2011fuzzy}
H.-J. Zimmermann, Fuzzy Set Theory—And Its Applications, Springer Science \& Business Media, 2011.

\bibitem{definitionoffuzzynumberklir1996fuzzy}
G.~J. Klir, B.~Yuan, Fuzzy sets and fuzzy logic: theory and applications, Prentice Hall, 1995.

\bibitem{fuzzynumberarithchen1985operations}
S.-H. Chen, Operations on fuzzy numbers with function principal, Tamkang Journal of Management Sciences 6~(1) (1985) 13--25.

\bibitem{rudin1976principles}
W.~Rudin, Principles of Mathematical Analysis, 3rd Edition, McGraw-Hill, 1976.

\bibitem{linearprogrammingmodel}
G.~B. Dantzig, M.~N. Thapa, Linear programming: 2: theory and extensions, Springer, 2003.

\bibitem{ordinalviolation}
B.~Golany, M.~Kress, A multicriteria evaluation of methods for obtaining weights from ratio-scale matrices, European Journal of Operational Research 69~(2) (1993) 210--220.

\bibitem{Prominenceeffect}
A.~Tversky, S.~Sattath, P.~Slovic, Contingent weighting in judgment and choice, Psychological Review 95~(3) (1988) 371.

\bibitem{Inputbasedconsistency}
S.~Guo, Z.~Qi, A fuzzy best-worst multi-criteria group decision-making method, IEEE Access 9 (2021) 118941--118952.

\bibitem{F-SBWM}
M.~Amiri, M.~Hashemi-Tabatabaei, M.~Keshavarz-Ghorabaee, A.~Kaklauskas, E.~K. Zavadskas, J.~Antucheviciene, A fuzzy extension of simplified best-worst method ({F-SBWM}) and its applications to decision-making problems, Symmetry 15~(1) (2022) 81.

\bibitem{GP-FBWM}
O.~Rostami, M.~Tavakoli, A.~Tajally, M.~GhanavatiNejad, A goal programming-based fuzzy best--worst method for the viable supplier selection problem: a case study, Soft Computing 27~(6) (2023) 2827--2852.

\bibitem{3point}
I.~O. Pappas, A.~G. Woodside, Fuzzy-set qualitative comparative analysis ({fsQCA}): guidelines for research practice in information systems and marketing, International Journal of Information Management 58 (2021) 102310.

\end{thebibliography}

\newpage
\appendix
\section{Data Collection Questionnaire}
\label{app:questionnaire}

\textbf{Objective:} This survey aims to evaluate and rank 20 potential warehouse locations (comprising both Local and Upcountry Depots) in Gujarat based on 5 logistical criteria. The evaluation is conducted by a Senior Commercial executive and is divided into three distinct parts to accommodate the proposed Linear $\alpha$-PFBWM framework.

\vspace{0.3cm}
\textbf{Part I: Evaluation of Criteria}

\vspace{0.2cm}

\textit{Instruction:} Please identify the most important (Best) and least important (Worst) criterion among the five given criteria. 
\subsection{Selection of Alternatives and Criteria}
Following a preliminary geographic and industrial screening, 20 logistics hubs and industrial estates across Gujarat were selected as potential warehouse locations. The set of alternatives is defined as $A = \{A_1, A_2, \dots, A_{20}\}$, where the locations correspond to: $A_1$: Aslali, $A_2$: Changodar, $A_3$: Sanand, $A_4$: Bavla, $A_5$: Naroda, $A_6$: Vapi, $A_7$: Ankleshwar, $A_8$: Hazira, $A_9$: Palsana, $A_{10}$: Dahej, $A_{11}$: Halol, $A_{12}$: Savli, $A_{13}$: Makarpura, $A_{14}$: Shapar Veraval, $A_{15}$: Metoda, $A_{16}$: Morbi, $A_{17}$: Sikka, $A_{18}$: Mundra, $A_{19}$: Kandla, and $A_{20}$: Gandhidham.

The DM evaluated these 20 sites against five primary criteria ($n = 5$):
\begin{itemize}
    \item \textbf{$C_1$ (Transportation Connectivity):} Proximity to national highways, railway freight corridors, and ports, directly influencing inbound and outbound freight efficiency.
    \item \textbf{$C_2$ (Real Estate and Setup Cost):} Expenses related to land acquisition, leasing, structural construction, and initial operational deployment.
    \item \textbf{$C_3$ (Proximity to Market and Dealers):} Distance to major consumer clusters and established distributor networks to minimize delivery lead times.
    \item \textbf{$C_4$ (Infrastructure Quality):} Availability and reliability of core utilities, including uninterrupted power, water supply, and internal road networks within the industrial estates.
    \item \textbf{$C_5$ (Labor Availability and Safety):} Access to skilled and semi-skilled labor pools, alongside safety and environmental compliance necessary for chemical and paint storage.
\end{itemize}

\textbf{Best Criterion ($c_B$):} \rule{3cm}{0.15mm} \hspace{1cm} \textbf{Worst Criterion ($c_W$):} \rule{3cm}{0.15mm}
\newline

\vspace{0.3cm}

\textit{Instruction:} Utilizing the 1--9 linguistic scale presented in Table~\ref{tab:linguistic_tfns} (ranging from Equally Preferred to Absolutely Preferred), please compare the Best criterion against all other criteria, and compare all remaining criteria against the Worst criterion.

\begin{table}[htbp]
\centering
\renewcommand{\arraystretch}{1.2}
\begin{tabular}{|l|c|c|c|c|c|}
\hline
\textbf{Criteria} & $C_1$ & $C_2$ & $C_3$ & $C_4$ & $C_5$ \\ \hline
\textbf{Best-to-Others (BO)} & & & & & \\ \hline
\textbf{Others-to-Worst (OW)} & & & & & \\ \hline
\end{tabular}
\end{table}

\vspace{0.3cm}
\textbf{Part II: Initial Performance Rating of Alternatives (Three-Point Estimation)}

\vspace{0.2cm}

\textit{Instruction:} Please rate the preliminary performance of all 20 warehouse locations against each of the 5 criteria on a continuous scale of 1 to 100. Instead of a single crisp value, provide three estimates to capture uncertainty: Minimum score (Min), Most Likely score (Likely), and Maximum score (Max).

\begin{table}[htbp]
\centering
\renewcommand{\arraystretch}{1.2}
\resizebox{\textwidth}{!}{
\begin{tabular}{|l|c|c|c|c|c|}
\hline
\textbf{Alternatives} & \textbf{$C_1$ (Min, Likely, Max)} & \textbf{$C_2$ (Min, Likely, Max)} & \textbf{$C_3$ (Min, Likely, Max)} & \textbf{$C_4$ (Min, Likely, Max)} & \textbf{$C_5$ (Min, Likely, Max)} \\ \hline
$A_1$: Aslali & & & & & \\ \hline
$A_2$: Changodar & & & & & \\ \hline
$A_3$: Sanand & & & & & \\ \hline
$A_4$: Bavla & & & & & \\ \hline
$A_5$: Naroda & & & & & \\ \hline
$A_6$: Vapi & & & & & \\ \hline
$A_7$: Ankleshwar & & & & & \\ \hline
$A_8$: Hazira & & & & & \\ \hline
$A_9$: Palsana & & & & & \\ \hline
$A_{10}$: Dahej & & & & & \\ \hline
$A_{11}$: Halol & & & & & \\ \hline
$A_{12}$: Savli & & & & & \\ \hline
$A_{13}$: Makarpura & & & & & \\ \hline
$A_{14}$: Shapar-Veraval & & & & & \\ \hline
$A_{15}$: Metoda & & & & & \\ \hline
$A_{16}$: Morbi & & & & & \\ \hline
$A_{17}$: Sikka & & & & & \\ \hline
$A_{18}$: Mundra & & & & & \\ \hline
$A_{19}$: Kandla & & & & & \\ \hline
$A_{20}$: Gandhidham & & & & & \\ \hline
\end{tabular}
}
\end{table}

\vspace{0.3cm}

\textbf{Part III: Parsimonious Evaluation of Reference Alternatives}

\vspace{0.2cm}

\textit{Instruction:} Based on the initial ratings provided in Part II, a distinct set of 5 reference alternatives (spanning from lowest to highest initial performance) has been extracted for \textbf{each} criterion individually. 

For each specific criterion, please identify the Best ($A_b$) and Worst ($A_w$) alternative among its extracted reference set. Then, provide the Best-to-Others and Others-to-Worst pairwise comparisons using the 1 to 9 linguistic scale in Table~\ref{tab:linguistic_tfns}. 

\begin{table}[htbp]
\textit{Questionnaire Format for Criterion $C_1$ (Reference Set: $A_{14}, A_{12}, A_1, A_6, A_{18}$):}
\vspace{0.2cm}

\centering
\renewcommand{\arraystretch}{1.2}
\begin{tabular}{|l|c|c|c|c|c|}
\hline
\textbf{Reference Alternatives} & $A_{14}$ & $A_{12}$ & $A_1$ & $A_6$ & $A_{18}$ \\ \hline
\textbf{Best-to-Others ($C_1$)} & & & & & \\ \hline
\textbf{Others-to-Worst ($C_1$)} & & & & & \\ \hline
\end{tabular}
\end{table}

\vspace{0.3cm}
\begin{table}[htbp]
\textit{Questionnaire Format for Criterion $C_2$ (Reference Set: $A_8, A_{10}, A_5, A_9, A_{14}$):}
\vspace{0.2cm}

\centering
\renewcommand{\arraystretch}{1.2}
\begin{tabular}{|l|c|c|c|c|c|}
\hline
\textbf{Reference Alternatives} & $A_8$ & $A_{10}$ & $A_5$ & $A_9$ & $A_{14}$ \\ \hline
\textbf{Best-to-Others ($C_2$)} & & & & & \\ \hline
\textbf{Others-to-Worst ($C_2$)} & & & & & \\ \hline
\end{tabular}
\end{table}

\vspace{0.3cm}
\begin{table}[htbp]
\textit{Questionnaire Format for Criterion $C_3$ (Reference Set: $A_{18}, A_{16}, A_{11}, A_3, A_5$):}
\vspace{0.2cm}

\centering
\renewcommand{\arraystretch}{1.2}
\begin{tabular}{|l|c|c|c|c|c|}
\hline
\textbf{Reference Alternatives} & $A_{18}$ & $A_{16}$ & $A_{11}$ & $A_3$ & $A_5$ \\ \hline
\textbf{Best-to-Others ($C_3$)} & & & & & \\ \hline
\textbf{Others-to-Worst ($C_3$)} & & & & & \\ \hline
\end{tabular}
\end{table}

\vspace{0.3cm}

\begin{table}[htbp]
\textit{Questionnaire Format for Criterion $C_4$ (Reference Set: $A_{14}, A_5, A_{16}, A_7, A_3$):}
\vspace{0.2cm}

\centering
\renewcommand{\arraystretch}{1.2}
\begin{tabular}{|l|c|c|c|c|c|}
\hline
\textbf{Reference Alternatives} & $A_{14}$ & $A_5$ & $A_{16}$ & $A_7$ & $A_3$ \\ \hline
\textbf{Best-to-Others ($C_4$)} & & & & & \\ \hline
\textbf{Others-to-Worst ($C_4$)} & & & & & \\ \hline
\end{tabular}
\end{table}

\vspace{0.3cm}
\begin{table}[htbp]
\textit{Questionnaire Format for Criterion $C_5$ (Reference Set: $A_{17}, A_{19}, A_5, A_2, A_6$):}
\vspace{0.2cm}

\centering
\renewcommand{\arraystretch}{1.2}
\begin{tabular}{|l|c|c|c|c|c|}
\hline
\textbf{Reference Alternatives} & $A_{17}$ & $A_{19}$ & $A_5$ & $A_2$ & $A_6$ \\ \hline
\textbf{Best-to-Others ($C_5$)} & & & & & \\ \hline
\textbf{Others-to-Worst ($C_5$)} & & & & & \\ \hline
\end{tabular}
\end{table}

\end{document}